\input ./style/arxiv-general.cfg
\documentclass[aop,MSNbibl,dvips]{arximspdf}
\makeatletter
   \@ifpackageloaded{graphicx}{}{\usepackage{graphicx}}
\makeatother
\doi{10.1214/14-AOP928} %kopijuoti is PTS
\volume{43}
\issue{4}
\pubyear{2015}
\firstpage{2084}
\lastpage{2118}
\makeatletter
\def\mid{\vert}

\newcommand{\eqref}[1]{(\ref{#1})}
\newtheorem{theorem}{Theorem}[section]
\newtheorem{lemma}[theorem]{Lemma}
\newtheorem{proposition}[theorem]{Proposition}
\newtheorem{corollary}[theorem]{Corollary}
\newproclaim{remark}[theorem]{Remark}
\newproclaim{fact}[theorem]{Fact}

\def\eps{\varepsilon}
\def\PP{\mathbb P}
\def\FF{\mathcal F}
\def\EE{\mathbb{E}}
\def\RR{\mathbb{R}}
\def\PH{\mathbb{R}^2_+}
\def\FR{\operatorname{Fr}}
\def\HH{\mathbb{H}^2 }
\def\BB{\mathcal{B} }
\def\MM{\mathcal{M} }
\makeatother

\begin{document}
\begin{frontmatter}

\title{Diffusion-limited aggregation on the hyperbolic~plane}
\runtitle{DLA on the hyperbolic plane}

\begin{aug}
\author[A]{\fnms{Ronen}~\snm{Eldan}\corref{}\ead[label=e1]{roneneldan@gmail.com}}
%\and
%\author{\fnms{}~\snm{}}
\runauthor{R. Eldan}
\affiliation{Microsoft Research}
%\dedicated{}
\address[A]{Microsoft Research\\
1 Microsoft way\\
Building 99\\
Redmond, Washington 98052\\
USA\\
\printead{e1}} %adresu isvedimo komanda gale!
%\address{}
\end{aug}

% HISTORY:
\received{\smonth{10} \syear{2013}}
\revised{\smonth{3} \syear{2014}}
%\accepted{\smonth{} \syear{}}

% ABSTRACT
%
\begin{abstract}
We consider an analogous version of the diffusion-limited aggregation
model defined on the hyperbolic plane. We prove that almost surely the
aggregate viewed at time infinity will have a positive density.
\end{abstract}

% KEYWORDS
\begin{keyword}[class=AMS]
%\kwd[Primary ]{}
\kwd{60K40}
\kwd{60K99}
%\kwd[; secondary ]{}
\end{keyword}
\begin{keyword}
\kwd{Random cluster growth models}
\kwd{diffusion-limited aggregation}
\kwd{hyperbolic space}
\kwd{harmonic measure}
\end{keyword}

\end{frontmatter}

%s1 #&#
\section{Introduction}\label{sec1}

The celebrated \emph{Diffusion-limited aggregation} (in short, \textit
{DLA}) model is a probabilistic model where particles undergoing a
certain diffusion stick together and form up into clusters. Most
commonly, the aggregate begins with a single particle at a fixed point,
and in every iteration a new particle arrives via a Brownian motion (or
some random walk) starting from infinity and stops at the moment it
hits the existing cluster, thus expanding it. This model was first
introduced by Witten and Sandler \cite{WS} in 1981 as a model which
could be used to represent several physical phenomena related to
systems where the principle mean of transport of particles is by
diffusion. Some examples of systems which appear to have DLA-like
behavior are electro-deposition, mineral deposits, and dielectric
breakdown systems.

The most interesting settings for the DLA model are naturally the two-
and three-dimensional Euclidean spaces (or the grids $\mathbb{Z}^2$ and
$\mathbb{Z}^3$). In these spaces, determining some of the most basic
properties of this model seem to be notoriously hard problems. For
example, it is not known whether the rate of growth of the diameter of
the aggregate is not $O(n^{1/d})$ where $n$ is the number of particles
and $d$ is the dimension, or whether or not the density of the cluster
at time infinity is zero. It is conjectured by physicists that the
answers to both these questions are positive. One of the only known
facts about DLA in Euclidean space is the result of Kesten~\cite{K},
who obtained the upper bound $O(n^{2/\max(d,3)})$ for the speed of
growth of the diameter of the DLA in $\mathbb{Z}^d$.
%We would also like
%to mention a paper of Barlow, Pemantle and Perkins \cite{BPP} in
%which the DLA model on a tree is studied.
We would also like to mention a paper of Barlow, Pemantle and Perkins \cite{BPP}
in which the DLA
model on a tree is studied as well as the work of Ebertz-Wagner \cite{naujas} in which it is shown that the
Euclidean DLA cluster will almost surely have infinitely many holes.

Roughly speaking, an analogous version of this model can be defined in
any space where the notion of diffusion exists. If the Poisson boundary
consists of one point (or, in other words, the definition of ``a
particle released at infinity'' makes sense) and the diffusion is
recurrent, the growth process can be defined so that law of the
location of a new particle is the harmonic measure of the existing
aggregate with pole at infinity. If the diffusion is transient (such as
in the case of $\mathbb{Z}^3$), one can consider the harmonic measure
with a pole far away from the aggregate, let the pole go to infinity
and take limits (i.e., conditioning on a random walk coming from
infinity to hit the cluster).

Another way to define the law of growth in settings where the diffusion
is transient is to use the \emph{time-reversibility} property of the
random walk. According to this property, the harmonic measure of a set,
with pole at infinity, is proportional to the so-called \emph
{equilibrium measure} associated to the set. For sets with sufficient
smoothness properties, this measure is absolutely continuous with
respect to the Hausdorff measure on the boundary of the set and its
density is proportional to the gradient, in the normal direction to the
boundary, of the solution of the Dirichlet problem with boundary
conditions $1$ on the set and $0$ at infinity. From a probabilistic
point of view, this density is roughly proportional to the probability
that a particle released close to the boundary of the set reaches
infinity before hitting the aggregate. Fortunately, this definition
also makes sense in settings where the Poisson boundary consists of
more than one point. A more detailed description of this will be given
in the next section.

Our aim in this paper is to study a DLA model defined on the \emph
{hyperbolic plane}, showing that in this case, the cluster at time
infinity almost surely admits a positive upper density. Our results
suggest that in the hyperbolic setting the behavior of the aggregate is
simpler to analyze than the Euclidean one. However, simulations point
that its geometry is still fairly complicated: it seems that the
so-called ``rich-get-richer'' behavior takes place also in this setting
and the aggregates look far from having a certain limit shape. Our
results may therefore be viewed as a modest attempt to rigorously study
certain properties of a model whose complexity is somewhat similar to
that of the Euclidean DLA. Diffusion-limited growth on general
Riemannian manifolds and specifically on the hyperbolic plane was
already considered in the physics literature, see \cite{CCB}; The
physical motivation for this study is that natural phenomena of
DLA-like behavior such as mineral dendrites, cell colonies and
cancerous tumors usually grow on curved surfaces.

In our construction, the particles will be metric balls of radius $1$.
We define $A_0$ to be a fixed point $p_0$ and recursively $A_{i+1} =
A_i \cup\{x\}$ where the point $x$ (thought of as the center of a
disc-shaped particle) will be picked from the set of points whose
distance from $A_i$ is exactly $2$ (which means exactly that the
corresponding discs will be tangent to each other) and will be
distributed in this set proportionally to the probability of escape to
infinity, described in the previous paragraph. We will also write
$A_\infty= \bigcup_{i=1}^\infty A_i$. The precise construction appears in
the next section.
Figure \ref{fig:hypdla} shows an instance of this construction drawn on
the Poincar\'e disc.

%f1 #&#
\begin{figure}

\includegraphics{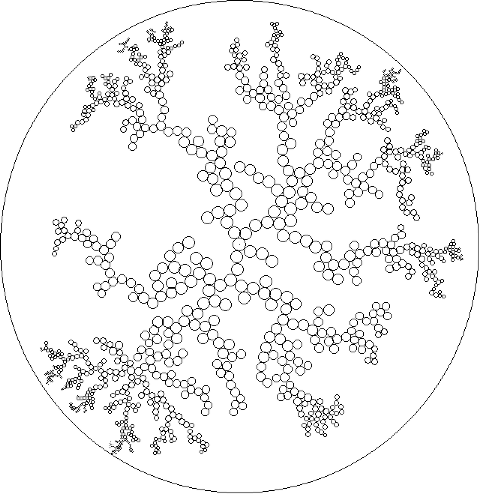}

\caption{A simulation of the DLA model with 1000 particles,
viewed on the Poincar\'e disc model.}  \label{fig:hypdla}
\end{figure}

In a metric measure space $X$ whose diameter is infinite, we say that a
locally-finite set $A \subset X$ has an \emph{upper density} greater or
equal to $c$ if there exists a point $p \in X$ and a sequence $R_1 <
R_2 < \cdots$ such that $R_i \to\infty$ as $i \to\infty$, such that
\[
\# \bigl(A \cap B(p, R_i) \bigr) \geq c \mu\bigl(B(p,
R_i)\bigr)\qquad \forall i \in\mathbb{N},
\]
where $B(p,r)$ is a metric ball centered at $p$ with radius $r$ and
$\mu
$ is the measure defined on $X$. We can use this definition
in the hyperbolic plane, using the standard hyperbolic distance as a
metric and the standard Riemannian volume of a set as a measure.

Our main theorem reads the following.

%th1.1 #&#
\begin{theorem} \label{mainthm}
The set $A_\infty= \bigcup_{i=1}^\infty A_i$ almost surely has an
upper density greater than $c$, where $c>0$ is a universal constant.
\end{theorem}

%re1.2 #&#
\begin{remark}
The reader may suspect that the above theorem follows from a general
geometric fact about the hyperbolic plane and does not use any of the
randomness in the model. Alas, there is an example of a connected set
which is a union of balls of radius 1, whose convex hull is the entire
plane, but whose upper density is zero. Indeed, consider the following
``spiral'' set: take a point $p \in\HH$ and $\theta_0 \in T_p$ (where
$T_p$ is the tangent space at $p$) and consider the exponential map $e\dvtx
T_p \to\HH$. Define
\[
A = \bigcup_{\theta\in[0, \infty)} B_H\bigl(
\exp_p \bigl(X_\theta R(\theta )\bigr), 1\bigr),
\]
where $X_\theta$ is a unit vector in $T_p$ whose angle with $\theta_0$
is $\theta$, $B_H(p,r)$ is a geodesic ball of radius $r$ centered at
$p$ and $R(\theta)$ is an increasing function. It is not hard to verify
that if the function $R(\theta)$ goes to infinity fast enough, the set
$A$ will have the properties described above.
\end{remark}

In the vaguest sense, the intuition behind the fact that the behavior
of the DLA model in the hyperbolic plane is different from the
conjectured behavior in Euclidean space is related to the rate of decay
of the harmonic potential. Consider two particles located at distance
$L$ apart. The probability for two Brownian paths released from the two
particles to intersect at some point is exponentially decreasing with
$L$ which, in turn, roughly means that when growing an aggregate from
those two points simultaneously, these two aggregates will hardly
interact. In particular, the new particles added to any two given
``arms'' of our aggregate will grow farther away from each other at
linear speed. This means that the growth law of the aggregate is almost
``local'' in the sense that the subtree related to each new particle
added to the aggregate will only ever be affected by its immediate
neighborhood and, moreover, their interaction will decrease
exponentially with time. The absence of long-range interactions will
prevent the multiscale phenomena, expected in the Euclidean case, from
occurring in our case.

Specifically, the geometry of the hyperbolic plane makes it much harder
to isolate certain parts of the DLA and disallowing them to grow
further by creating \textit{fjords} which are too narrow for particles to
come through, which in turn means that the DLA will locally keep
growing at most of its parts and will eventually fill the whole space.

Let us now review the general plan of our proof, while trying to
explain how the aforementioned properties of hyperbolic geometry come
into play.

The main step of the proof will be to show that there exists a
universal constant $R_0>0$ such that for any metric ball $B$ of radius
$R_0$, there is a probability of at least $0.99$ that the aggregate
will intersect this ball, no matter how far the ball is from the
starting point of the aggregate.

The proof of this step relies heavily on the fact that the upper
half-plane, $\RR\times(0, \infty)$, is isometric to $\HH$ via a
conformal mapping (using the so-called Poincar\'e metric). Regarding
our aggregate on the upper half-plane and choosing the correct
embedding, this is easily reduced to showing that an aggregate which
begins at the point $(0,\eps)$ reaches, with a nonnegligible
probability, any rectangle of the form $\Psi= [-C, C] \times[1,2]$
where $C>0$ is a universal constant and $\eps$ is an arbitrarily small
positive number.

At this point, let us now try to further illustrate the difference
between Euclidean and hyperbolic geometry which we are going to
exploit: in order for the aggregate to never reach the rectangle $\Psi
$, it has to encompass $\Psi$, at least in the sense that $\Psi$ will
be contained in the convex hull of the aggregate before any point of
the $\Psi$ has a chance to be reached by it. In particular, the
aggregate has to reach one of the lines $\{x=\pm C\}$. Now, note that
any geodesic line connecting the starting points with these two lines
actually passes through $\Psi$. In other words, the rectangle $\Psi$
acts as bottleneck which prevents the aggregate from encompassing it.
It is easy to see that no analogous phenomenon takes place in the
Euclidean space.

%re1.3 #&#
\begin{remark}
As mentioned above, in the paper of Barlow, Pemantle and Perkins \cite
{BPP}, a diffusion-limited aggregation on an infinite regular tree is
studied. The fact that the hyperbolic space has a tree-like structure
may mislead the reader to think that the model studied in their paper
is closely related to our model, and that the two are therefore
expected to behave in the same way. While these two models are
superficially similar and both called DLA, their behavior is
nevertheless quite different. Remark that on the discrete tree, each
connected component of the complement of a given subtree looks exactly
the same. Thus, the tree counterpart of our process would be defined
such that the rate of growth of the aggregate is constant on all points
of its boundary, regardless of its geometry. By definition, this
aggregate will eventually fill the entire tree and it is not hard to
see that it would do it in a rather uniform way.
\end{remark}

Let us try to explain our strategy to formally establish the fact that
$\Psi$ is likely to be reached by the aggregate before one of the lines
$\{x=\pm C\}$ is reached.

The idea will be to establish bounds on the rate of growth of the
minimum encompassing rectangle of the aggregate, hence the maximal
$x$-coordinate of the aggregate at time $t$, denoted by $X(t)$, and the
maximal $y$-coordinate, denoted by $Y(t)$ (see Figure~\ref{figdef}
below). In order to prove that the aggregate reaches the rectangle~$\Psi
$, it will be enough to show that $X(t)$ does not grow much faster than
$Y(t)$. We will work with a continuous time $t \in\RR^+$, so that
growth of the cluster is according to an exponential clock whose rate
is proportional to the capacity, which ensures us that in small time
intervals the expected rate of growth in different parts of the cluster
is roughly independent (this is defined in Section~\ref{sec2}).

%
%f2 #&#
\begin{figure}[b]

\includegraphics{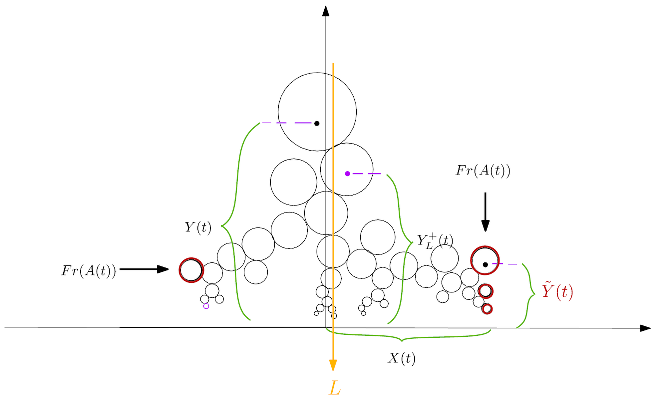}

\caption{The definitions $X(t), Y(t), Y_L^+(t), \FR(A(t))$
and $\tilde Y(t)$ illustrated.} \label{figdef}
\end{figure}

Two key geometric lemmas proven in Section~\ref{sec3} will provide an upper
bound for the rate of growth of $X(t)$ and a lower bound for the rate
of growth of $Y(t)$. The former bound, whose proof uses the easy fact
that in the half-plane model the $y$ coordinate of the center of metric
circle of radius $1$ is proportional to its Euclidean radius, roughly
says that $\frac{d}{dt} \EE[X(t)] < C Y(t)$. According to the latter
bound, which makes use of the conformal invariance, the probability of
$Y(t)$ to multiply itself by a constant during a unit time interval is
at least of the order $c Y(t) / (X(t) + Y(t))$ or, in other words,
roughly $d Y(t) > c Y^2(t) / X(t)$. Here, $c,C$ are universal constants.

Next, we note that (very informally) these bounds combined give
\[
d \frac{X(t)}{Y(t)} = \frac{d X(t)}{Y(t)} - \frac{X(t) \,d Y(t)}{Y(t)^2} \leq C - c.
\]
One would expect that by integrating those two bounds it should be
possible to attain an estimate of the form $Y(t) > X(t)^{\alpha}$ where
$\alpha$ is a positive constant which depends on the ratio $C/c$, at
least in expectation. However, it seems like the above bounds cannot be
pushed to give constants which would yield $\alpha\geq1$.

Because of this, we have to do something a little more complicated. We
define $\tilde Y(t)$ as the height of the cluster close to the edge
where $x$ attains its maximum (as in Figure~\ref{figdef}), and
consider two different cases: if $\tilde Y(t)$ is much smaller than
$Y(t)$, we get that $d X(t)$ is small enough so that the two bounds
above can be integrated to attain that $d \frac{X(t)}{Y(t)}$ is
negative. On the other hand, if $\tilde Y(t)$ and $Y(t)$ are
comparable, it turns out that we expect $X(t) / Y(t)$ to decrease due
to a completely different reason (provided that it is not too small).
We know that there is a nonnegligible probability that the height of
the cluster will grow rather rapidly close to its edge [hence close to
the place where $X(t)$ is attained] and, therefore, $Y(t)$ can multiply
itself by a constant within a constant amount of time. All of this is
carried out in Section~\ref{sec4}.

Once we have those two bounds, which can be combined into a unified
bound on the (expected) rate of growth of $R(t) = X(t) / Y(t)$ the
proof of the main step is just a matter of defining the correct
martingale and using the optional stopping theorem. Note, however, that
the process $X(t) / Y(t)$ cannot actually be a super-martingale as we
know that it is always positive, and it clearly does not converge.
Ideologically, this process should be regarded as a super-martingale
reflecting at zero, and for such processes, the optional stopping
theorem cannot help (it is not hard to see that Brownian motion with a
strong drift toward zero and reflection at zero can be almost surely
stopped at arbitrarily large values with a stopping time of finite
expectation). With a little extra work, we show that the process $x \to
R(\min\{t; X(t) > x \})$ is also a super-martingale with reflection at
zero and a strong enough drift, which turns out to be enough. In
Section~\ref{sec5}, we tie up the loose ends, showing how the main step can be
used to complete the proof.

%s2 #&#
\section{Preliminaries}\label{sec2}

%s2.1 #&#
\subsection{The Poincar\'{e} half-plane model}
We denote the hyperbolic plane by $\HH$. For two points $p_1,p_2 \in
\HH
$, we define the hyperbolic distance between them by $d_H(p_1,p_2)$. In
many cases, we will view the hyperbolic plane using the Poincar\'{e}
half-plane model, which is the usual open half plane $\PH:= \RR\times
(0, \infty)$ (sometimes called the Poincar\'{e} half-plane) equipped
with an embedding $H\dvtx \PH\to\HH$ and a distance function defined by
%
%e1 #&#
\begin{eqnarray}
\label{poincare} d_H\bigl((x_1,y_1),
(x_2,y_2)\bigr) &=& d_H\bigl(H(x_1,y_1),
H(x_2,y_2)\bigr)
\nonumber
\\[-8pt]
\\[-8pt]
\nonumber
& =&
\operatorname{Arcosh} \biggl(1+ \frac{(x_2-x_1)^2+(y_2-y_1)^2}{2 y_1 y_2} \biggr).
\end{eqnarray}
By slight abuse of notation, throughout this note we will sometimes
allow ourselves to interchange freely between the roles of $p$ and
$H(p)$, whenever the intention is clear from the context.

For a point $p \in\PH$, let $B_H(p, r) \subset\PH$ be the closed
$d_H$-ball centered at $p$ with radius $r$ and let $B_E(p, r) \subset
\PH$ be the closed Euclidean-ball centered at $p$ with radius $r$. We
will often use the following elementary estimate, which follows
immediately from formula (\ref{poincare}).

%le2.1 #&#
\begin{lemma} \label{ballslem}
For any $(x,y) \in\PH$, one has
\[
B_E\bigl((x,y), 0.5 y\bigr) \subseteq B_H\bigl((x,y),
1\bigr) \subseteq B_H\bigl((x,y), 2\bigr) \subseteq B_E
\bigl((x,y), 7 y\bigr).
\]
\end{lemma}

Another basic fact of which we will make use quite often is the
invariance of the model to M\"{o}bius transformations leaving $\PH$ intact:

%fa2.2 #&#
\begin{fact} \label{invariance}
For any constants $\alpha\in\RR$ and $\beta>0$ consider the transformation
\[
T\dvtx (x,y) \to(\beta x + \alpha, \beta y).
\]
Then $d_H$ is invariant under $T$, namely,
\[
d_H\bigl((x_0, y_0), (x_1,
y_1)\bigr) = d_H\bigl(T(x_0,
y_0), T(x_1, y_1)\bigr)
\]
for all $(x_0,y_0), (x_1,y_1) \in\PH$.
\end{fact}

We denote by $\HH(\infty)$ the set of ideal points (or omega points) of
the hyperbolic plane. We also define
\[
\PH(\infty) = \RR\times\{0\} \cup\{\infty\}.
\]
By continuity, we can extend an embedding $H\dvtx \PH\to\HH$ to the set
$\PH(\infty)$.

One last property of the Poincar\'e model which we will exploit is its
\emph{conformality}, namely, the fact that the map $H\dvtx \HH\to\PH$ is a
conformal map. Thanks to this fact and since, according to a theorem of
P. L\'{e}vy, the path of a Brownian motion is invariant under conformal
maps, we have the following.

%fa2.3 #&#
\begin{fact}[(Conformal invariance)] \label{conformal}
Let $A \subset\HH$ be a measurable set and let $x \in\HH$ be any
point. The path of a hyperbolic Brownian motion starting at $x$ and
stopped when it reaches $A \cup\HH(\infty)$ has the same distribution
as the image under the map $H$ of the path of the usual Euclidean
Brownian motion defined on $\PH$ started at $H^{-1}(x)$ and stopped at
$H^{-1}(A) \cup\PH(\infty)$.
\end{fact}

%s2.2 #&#
\subsection{The harmonic measure}

As explained above, in Euclidean space, the DLA is usually defined via
particles arriving from infinity, or equivalently, the place of the
particle added to the aggregate is has a distribution whose law is the
harmonic measure on the boundary of the existing aggregate, with a pole
at infinity. Unfortunately, in the hyperbolic space, there is no
natural analogous definition, as the harmonic measure actually depends
on the point in $H(\infty)$ from which the particle is released (or, in
other words, the Poisson boundary contains more than one point). In
order to find a definition of a DLA growth model on the hyperbolic
plane that makes sense, we use the following fact which is a
consequence of the time reversibility of the Brownian motion (for a
proof, see \cite{IM}, page~252 and \cite{MP}, Theorem 8.33).

%fa2.4 #&#
\begin{fact} \label{timereverse}
For any smooth set $A \subset\RR^n$, $n \geq3$, there exists a
constant $C_A$ such that for any $x \in\partial A$, one has
\[
C_A m_{A, \infty} (x) = \lim_{\varepsilon\to0^+}
\frac{1}{\varepsilon} \PP \biggl({ {\mbox{A Brownian motion released from } x+\vec{n}
\varepsilon} \atop\mbox{ reaches $\infty$ before hitting $A$ } }\biggr),
\]
where $\vec{n}$ is the normal direction to $\partial A$ at $x$,
pointing outward and $m_{A, \infty}(x)$ is the density of the harmonic
measure of the domain $A$ with pole at $\infty$ evaluated at the point $x$.
\end{fact}

Fortunately, the right-hand side of the above formula can be defined
just the same in the hyperbolic plane. Fix two measurable subsets $A,B
\subset\HH\cup\HH(\infty)$ such that $\HH(\infty) \subset A \cup B$
and fix a point $x \in\partial A \setminus H(\infty)$ such that
$\partial A$ is smooth at $x$. Denote by $T_x$ be the tangent space of
$\HH$ at $x$ and let $v \in T_x$ be the outward normal to $\partial A$
at $x$. Consider the exponential map $\exp_{x}\dvtx T_x \to\HH$. We define
\[
m_{A, B} (x) = \lim_{\varepsilon\to0^+ } \frac{1}{\varepsilon} \PP \biggl(
{{\mbox{A brownian motion released from } \exp_x(\varepsilon v ) } \atop
\mbox{ reaches $B$ before hitting $A$ } }\biggr).
\]
For all measurable $D \subset\partial A \setminus\HH(\infty)$, we define
\[
\MM_{A, B} (D) = \int_{D} m_{A,B} (x) \,d
\ell(x),
\]
where $\ell(\cdot)$ is the standard length measure in the hyperbolic
plane. We claim that the above integral is well defined and finite
whenever $A$ is a finite union of metric balls. Indeed, the boundary of
such a set is smooth up to a finite set of points, which means that the
above integral is well defined. Moreover, it is evident from the above
definition that $m_{A,B}$ admits the following monotonicity property:
for two sets $A' \subset A$ such that $x \in\partial A' \cap\partial
A$, one has $m_{A',B} (x) \geq m_{A,B}(x)$. Consequently, the function
$m_{A,B}(x)$ is bounded on $\partial A$ and the integral is finite.

%re2.5 #&#
\begin{remark}
In fact, this definition is valid for any set whose boundary is a
rectifiable curve (see \cite{P}, Example 1.2).
\end{remark}

Finally, when $A \cap\HH(\infty) = \varnothing$, we also abbreviate
%
%e2 #&#
\begin{equation}
\label{harmonic} \MM_A(D) = \MM_{A, \HH(\infty)}(D).
\end{equation}
In view of Fact \ref{timereverse}, it seems natural to construct our
DLA cluster using this measure.

%s2.3 #&#
\subsection{Construction of the DLA}

The evolution of our aggregate will be represented via a sequence of
random finite sets $A_1 \subset A_2 \subset\cdots,$ each element of which
is a point in $\HH$ represents a single particle. The particles are
assumed to be metric balls of radius $1$, and the elements of the above
sets are the centers of those metric balls, hence the actual aggregate
takes the form
\[
\bigcup_{p \in A_i} B_H(p, 1).
\]
We fix a point $p_0 \in\HH$ which we regard as the origin of the
aggregate. We begin with the set $A_0 = \{ p_0 \}$. The set $A_{i+1}$
will be the existing aggregate $A_i$ with the addition of one point
representing the center of the new particle. In order to define the law
according to which this new point is distributed, we will need some
more definitions.

For a finite set $A \subset\HH$, we define
\[
\BB(A) = \bigcup_{x \in A} B_H(x, 2).
\]
The point of taking balls of radius 2 is that any ball centered at a
point in $\partial\BB(A)$ whose radius is $1$ will be tangent to the
aggregate (which is assumed to be a union of balls of radius 1). Define
\[
\mu_A(\cdot) = \operatorname{Cap}(A)^{-1} \MM_{\BB(A)}( \cdot),
\]
where
\[
\operatorname{Cap}(A):= \MM_{\BB(A)} \bigl(\partial\BB(A)\bigr)
\]
is a normalizing constant to which we will refer to as the \emph
{capacity} of $A$ and where the measure $\MM_{\BB(A)}$ is defined in
equation (\ref{harmonic}). Note that by definition, the measure $\mu_A$
is a probability measure.

%
%re2.6 #&#
\begin{remark}
The quantity $\operatorname{Cap}(A)$ is sometimes referred to as the inverse \emph
{Riemann modulus} of $A$. It is a well known fact, which is a
consequence of Schottky's theorem that it is invariant under conformal
maps of the hyperbolic plane.
\end{remark}

We can finally define by recursion,
\[
A_{i+1} = A_i \cup\{X_i\},
\]
where $X_i$ is a random point in $\partial\BB(A_i)$ distributed
according to the law $\mu_{A_i}$.

Throughout this note, we will usually allow ourselves to interchange
freely between $A_i$ and $H^{-1} (A_i)$ (when this does not cause any
confusion), thus sometimes considering $A_i$ as a subset of $\PH$.

%s2.4 #&#
\subsection{Continuous time}
In our proofs, it will be more convenient to regard our process in
continuous time. We define a sequence of times $t_0, t_1, t_2,\ldots $ by
the following inductive law: Define $t_0 = 0$, and for all $i \geq0$,
let $t_{i+1} - t_{i}$ be an exponentially-distributed variable whose
expectation is $\operatorname{Cap}(A_i)^{-1}$, independent from all the rest. Finally,
we define
\[
A(t) = A_{i(t)},
\]
where
\[
i(t) = \max\{i; t_i \leq t \}.
\]
We denote by $\FF_t$ the filtration corresponding to the process. The
next fact will be useful to us:

%fa2.7 #&#
\begin{fact} \label{newparticlefact}
The process $A(t)$ is a Markov process, hence for every random variable
$X$ measurable with respect to $\FF_\infty$ and every $t \geq0$,
\[
\EE[X| \FF_t] = \EE\bigl[X | A(t)\bigr].
\]
Moreover, for any $t$ and for any measurable $B \subset\partial\BB
(A(t))$, one has
%
%e3 #&#
\begin{equation}
\label{newparticle} \lim_{\varepsilon\to0^+}  \frac{1}{\varepsilon} \PP \bigl(
B \cap A(t + \varepsilon) \neq\varnothing\vert A(t) \bigr) = \MM_{\BB(A(t))}(B)
\end{equation}
and for all $B$ such that $B \cap\partial\BB(A(t)) = \varnothing$,
%
%e4 #&#
\begin{equation}
\label{newparticleout} \lim_{\varepsilon\to0^+} \frac{1}{\varepsilon} \PP \bigl(
B \cap \bigl(A(t + \varepsilon) \setminus A(t) \bigr) \neq\varnothing\vert A(t)
\bigr) = 0.
\end{equation}
\end{fact}

\begin{pf}
The Markov property follows immediately from the definition of the
process. In order to prove formula (\ref{newparticle}), we make note
that for all $i \in\mathbb{N}$,
\begin{eqnarray*}
\operatorname{Cap}(A_i) &\leq&\sum_{p \in A_i}
\MM_{\BB(A_i)} \bigl(\partial B_H(p, 2)\bigr)\\
& \leq&
\sum_{p \in A_i} \MM_{B_H(p, 2)} \bigl(\partial
B_H(p, 2)\bigr) = P_0 i
\end{eqnarray*}
for some constant $P_0 > 0$. Therefore, we can estimate
\begin{eqnarray*}
&&\PP\bigl(i(t + \varepsilon) \geq i(t) + 2 \vert A(t) \bigr) \\
&&\qquad\leq
\PP\bigl(t_{i(t)+1} \leq t + \varepsilon | A(t) \bigr) \PP
\bigl(t_{i(t)+2} < t_{i(t) +
1} + \varepsilon | A(t)\bigr) \\
&&\qquad\leq
\PP\bigl(E\bigl(1/\bigl(P_0 i(t)\bigr)\bigr) < \varepsilon\bigr) \PP
\bigl(E\bigl(1/\bigl(P_0 \bigl(i(t) + 1\bigr)\bigr)\bigr) < \varepsilon
\bigr) = O\bigl(\varepsilon^2\bigr),
\end{eqnarray*}
where $E(v)$ denotes an exponential variable with expectation $v$. We
deduce that the probability that more than one particle is added to the
cluster in an interval of the form $[t, t+\varepsilon]$ is of the order
$\varepsilon^2$. Since by definition, the next particle added must be at
$\partial\BB(A(t))$, equation (\ref{newparticleout}) follows.
Next, we have
\begin{eqnarray*}
&&\lim_{\varepsilon\to0^+} \frac{1}{\varepsilon} \PP \bigl(  A(t +
\varepsilon) \cap B \neq\varnothing \vert A(t) \bigr)\\
&&\qquad =
\PP( A_{i(t) + 1} \cap B \neq\varnothing) \lim_{\varepsilon\to0^+}
\frac
{1}{\varepsilon} \PP\bigl(t_{i(t) + 1} \leq t + \varepsilon| A(t) \bigr)\\
&&\qquad =
\mu_{A(t)}(B) \lim_{\varepsilon\to0^+} \frac{1}{\varepsilon} \bigl(1 -
\exp\bigl(- \varepsilon \operatorname{Cap}\bigl(A(t)\bigr) \bigr)\bigr) = \mu_{A(t)}(B) \operatorname{Cap}
\bigl(A(t)\bigr),
\end{eqnarray*}
which proves (\ref{newparticle}). The proof is complete.
\end{pf}

%s3 #&#
\section{Geometric lemmas}\label{sec3}
The goal of this section is to prove two geometric lemmas which will
serve as central ingredients in the proof. Throughout this section, we
assume that the embedding of $\HH$ in $\PH$ has been fixed, and
consider the aggregate $A(t)$ as a subset of $\PH$. We begin with some
definitions which will be frequently used later on.

For every time $t \geq0$, we define
\[
X(t) = \sup\bigl\{|x|; \exists y \mbox{ such that } (x,y) \in A(t) \bigr\}
\]
and
\[
Y(t) = \sup\bigl\{y; \exists x \mbox{ such that } (x,y) \in A(t) \bigr\}.
\]
We define also,
\[
Y_L^+(t) = \sup\bigl\{y; \exists x \geq L \mbox{ such that } (x,y)
\in A(t) \bigr\}
\]
and
\[
Y_L^-(t) = \sup\bigl\{y; \exists x \leq L \mbox{ such that } (x,y)
\in A(t) \bigr\}.
\]
For a particle $b \in A(t)$, we say that $b$ is in the \emph{front} of
$A(t)$ and denote $b \in\FR(A(t))$ if there exists a point
$p=(x,y) \in\PH$ having $d_H(b,p) \leq1$ and $|x| \geq X(t)$.
Finally, we define
\[
\tilde Y(t) = \sup\bigl\{y; (x,y) \in\FR\bigl(A(t)\bigr) \bigr\}.
\]
These definitions are illustrated in Figure~\ref{figdef}.

We begin with the following\vspace*{1pt} upper bound for the rate of growth of
$X(t)$, which turns out to be controlled by $\tilde Y(t)$ in expectation.

%le3.1 #&#
\begin{lemma} \label{lemX}
There exists a universal constant $C >0$ such that for all $t \geq0$,
one has almost surely
%
%e5 #&#
\begin{equation}
\label{lemxf} \lim_{\varepsilon\to0^+} \frac{1}{\varepsilon} \bigl(\EE
\bigl[X(t+\varepsilon ) | \FF_t\bigr] - X(t) \bigr) \leq C \tilde Y(t).
\end{equation}
\end{lemma}

The geometric intuition behind this lemma is the following: first of
all, by the nature of the harmonic measure, if each particle of the
aggregate would be allowed to duplicate itself with a constant rate,
regardless of the other existing particles, this would result in a
faster expected growth of $X(t)$. Consequently, it is enough to prove
this lemma for the simpler model in which the harmonic measure is
replaced with the usual length measure on the boundary of the
aggregate. By definition of the \emph{front} of the aggregate, we may
only consider particles in $\FR(A(t))$ since only these can cause
$X(t)$ to increase by duplicating. Lemma \ref{ballslem} shows us that a
particle whose height is $y$ is expected to duplicate to a particle at
horizontal distance $C y$ for some fixed $C>0$, which implies that the
total expected horizontal growth of the aggregate at unit time is
bounded by the sum $\sum_{p \in\FR(A(t))} C y(p)$. The geometry of the
front of the aggregate only allows a constant number of particles at a
given height, which will allow us to bound this sum by that of a
geometric sequence, which only depends on the largest summand. In other
words, the expected growth will be bounded by the height of $\FR(A(t))$.

We will first need the following intermediate, technical result, whose
proof is postponed to the end of the section.

%le3.2 #&#
\begin{lemma} \label{lemannoying}
For all $t \geq0$ and given any aggregate $A(t)$, there exist
constants $C, \eps_0 > 0$ such that for all $\eps< \eps_0$
\[
\PP \bigl(X(t + \eps) - X(t) > \alpha\mid\mathcal{F}_t \bigr) \leq
C \eps\min \bigl(\alpha^{-2}, 1 \bigr) \qquad\forall\alpha> 0.
\]
\end{lemma}

\begin{pf*}{Proof of Lemma \ref{lemX}}
Fix a time $t > 0$ and an aggregate $A(t)$. For all $s > 0$, define the set
\[
B_s = \bigl\{(x,y) \in\partial\BB\bigl(A(t)\bigr); |x| - X(t) \geq
s \bigr\}.
\]
According to formulas (\ref{newparticle}) and (\ref{newparticleout}),
one has
\[
\lim_{\varepsilon\to0^+} \frac{1}{\varepsilon} \PP \bigl(X(t + \varepsilon ) - X(t)
\geq s \mid\FF_t \bigr) = \MM_{\BB(A(t))}(B_s).
\]
Using Lemma \ref{lemannoying}, we know that there exist constants
$\eps
_0,C>0$ such that for all $\eps< \eps_0$,
\[
\int_{s=0}^\infty\frac{1}{\varepsilon} \PP \bigl(X(t +
\varepsilon) - X(t) \geq s \mid\FF_t \bigr) \,ds < C.
\]
Consequently, we may use the dominated convergence theorem to get
%
%e6 #&#
\begin{eqnarray}
\label{ineqexp}&& \lim_{\varepsilon\to0^+} \frac{1}{\varepsilon} \EE \bigl(X(t +
\varepsilon ) - X(t) \mid\FF_t \bigr)\nonumber\\
&&\qquad =
\lim_{\varepsilon\to0^+} \int_{s=0}^\infty
\frac{1}{\varepsilon} \PP \bigl(X(t + \varepsilon) - X(t) \geq s \mid\FF_t
\bigr) \,ds \\
&&\qquad=
\int_{s=0}^\infty\MM_{\BB(A(t))}(B_s)
\,ds.\nonumber
\end{eqnarray}
Next, using Lemma \ref{ballslem}, we learn that for two points
$(x_1,y_1), (x_2,y_2) \in\PH$ one has
%
%e7 #&#
\begin{equation}
d_H\bigl((x_1,y_1), (x_2,y_2)
\bigr) \leq2 \quad\Rightarrow\quad |x_1 - x_2| \leq C_1
y_1.
\end{equation}
It follows that, using the definition of $\FR(A(t))$,
\[
B_s \subset\mathop{\bigcup_{(x,y) \in\FR(A(t)) }}_{|x| + C_1 y \geq X(t) +
s } \partial\BB
\bigl( \bigl\{(x,y)\bigr\} \bigr)
\]
for all $s > 0$. Next, observe that for all $(x,y) \in\PH$, one has by
definition
%
%e8 #&#
\begin{equation}
\MM_{\BB(A(t))} \bigl(\partial\BB\bigl(\{x,y\}\bigr)\bigr) \leq
\MM_{\BB(\{x,y\})} \bigl(\partial\BB\bigl(\{x,y\}\bigr)\bigr) =: P_0,
\end{equation}
where $P_0>0$ is a universal constant [in particular, it does not
depend on $(x,y)$]. A combination of the two above equations teaches us that
\begin{eqnarray*}
\MM_{\BB(A(t))} (B_s)& \leq&\mathop{\sum_{(x,y) \in\FR(A(t)) }}_{|x| +
C_1 y
\geq X(t) + s }
\MM_{\BB(\{x,y\})} \bigl(\partial\BB\bigl(\{x,y\}\bigr)\bigr)\\
& \leq&
\# \bigl\{ (x,y) \in\FR\bigl(A(t)\bigr); |x| + C_1 y \geq X(t) + s
\bigr\} P_0 \\
&\leq&
\# \bigl\{ (x,y) \in\FR\bigl(A(t)\bigr); C_1 y \geq s \bigr\}
P_0.
\end{eqnarray*}
A combination of the above inequality with (\ref{ineqexp}) yields
%
%e9 #&#
\begin{eqnarray}
\label{triangle} &&\lim_{\varepsilon\to0^+} \frac{1}{\varepsilon} \EE \bigl(X(t +
\varepsilon ) - X(t) \mid\FF_t \bigr)\nonumber\\
&&\qquad \leq
P_0 \int_{s=0}^{\infty} \# \bigl\{ (x,y)
\in\FR\bigl(A(t)\bigr); C_1 y \geq s \bigr\} \,ds\\
&&\qquad =
P_0 C_1 \sum_{(x,y) \in\FR(A(t))} y.\nonumber
\end{eqnarray}
We turn to estimate the above sum. Recall the definition of $\FR
(A(t))$ and observe that Lemma \ref{ballslem} also implies
%
%e10 #&#
\begin{equation}
\label{fkeq} (x,y) \in\FR\bigl(A(t)\bigr) \quad\Rightarrow\quad X(t) - C_1
y \leq|x| \leq X(t).
\end{equation}
Now, for any number $K>0$, define
\[
F(K) = \bigl\{(x,y) \in\PH; X(t) - C_1 y \leq|x| \leq X(t) \mbox{
and } K / 2 \leq y \leq K \bigr\}.
\]
Fact \ref{invariance} teaches us that the hyperbolic volume of $F(K)$
does not depend on $K$, as a dilation of the number $K$
corresponds to rescaling of each connected component of $F(K)$ about a
point on the $x$-axis. Since these sets are compact and separated from
the $X$ axis, they have a finite volume. It is thus clear that the
cardinality of any set of disjoint $d_H$-balls of radius $1$ whose
centers are in $F(K)$ is bounded by some universal constant $C_2$
(which does not depend of $K$). Consequently,
%
%e11 #&#
\begin{equation}
\label{eqx2} \sum_{(x,y) \in\FR(A(t)) \cap F(K) } y \leq C_2 K.
\end{equation}
Note that by equation (\ref{fkeq}), we have
\[
\FR\bigl(A(t)\bigr) \subset\bigcup_{j=0}^\infty
F\bigl(\tilde Y(t) 2^{-j} \bigr).
\]
Using this fact with (\ref{triangle}) and (\ref{eqx2}) finally gives
\begin{eqnarray*}
&&\lim_{\varepsilon\to0^+} \frac{1}{\varepsilon} \EE \bigl(X(t + \varepsilon ) - X(t)
\mid\FF_t \bigr) \\
&&\qquad\leq P_0 C_1 \sum
_{(x,y) \in\FR(A(t))} y\\
&&\qquad =
P_0 C_1 \sum_{j=0}^\infty
\mathop{\sum}_{(x,y) \in\FR(A(t)) \cap
F(\tilde Y(t) 2^{-j} ) } y \leq2 P_0 C_1
C_2 \tilde Y(t)
\end{eqnarray*}
and the proof of the lemma is complete.
\end{pf*}

The next bound can be regarded as a lower bound for the rate of growth
of $Y(t)$, whose proof relies heavily on the conformity of the map $H$.
This bound is a consequence of a rather straightforward geometric fact
about the harmonic measure: given a rectangle of the form $K = [-M,M]
\times[0,1]$, consider the harmonic measure $\mathcal{M}_{K, \RR
\times\{ 0 \}}$ evaluated on different points of its upper edge
$[-M,M] \times\{1\}$. The density of this measure at a point $(x,1)
\in\partial K$ is bounded from below by $c (M - |x| + 1)^{-1}$. Recall
that, by definition, the aggregate $A(t)$ is contained in the rectangle
$[-X(t), X(t)] \times[0, Y(t)]$. This means that the probability of
the aggregate's top-most particle [the one attaining $Y(t)$] to
duplicate itself upward, and thus increase $Y(t)$ by a constant
multiplicative factor is bounded from below by $c Y(t) / (X(t) + Y(t))$.

We will need a bound that deals with a slightly more general scenario,
in which one has the additional information that a constant fraction of
the aggregate's height is attained at a point close to the front of the
aggregate, say located at $X(t) - L$. In this case, the above estimate
on the harmonic measure gives a rate of growth of $c L^{-1}$. However,
since we do not assume here that the aggregate is entirely contained in
the corresponding rectangle the argument will have to be slightly more delicate.

%le3.3 #&#
\begin{lemma} \label{lemY}
There exists a constant $c>0$ such that for all $t \geq0$ one has
%
%e12 #&#
\begin{equation}
\label{ygrowseq} \lim_{\varepsilon\to0^+} \frac{1}{\varepsilon} \PP\bigl[Y(t +
\varepsilon) > (1+c) Y(t) | A(t) \bigr] > c \frac{Y(t)}{Y(t) + X(t)}.
\end{equation}
Furthermore, for any constant $\Delta\geq1$, there exists a constant
$c(\Delta)$ (which depends only on $\Delta$) such that the following holds:
Let $L \in\RR$ and suppose that $Y(t) \leq\Delta Y_L^+(t)$. Then
%
%e13 #&#
\begin{eqnarray}
\label{ylp}&& \lim_{\varepsilon\to0^+} \frac{1}{\varepsilon} \PP
\bigl[Y_{L - 10Y(t)}^+(t + \varepsilon)
 \geq(1 + c) Y_L^+(t) | A(t)
\bigr]\nonumber
\\[-8pt]
\\[-8pt]
\nonumber
&&\qquad > c(\Delta) \frac{Y(t)}{Y(t)
+ X(t) - L}.
\end{eqnarray}
Likewise, if $Y(t) \leq\Delta Y_L^-(t)$ then
%
%e14 #&#
\begin{eqnarray}
\label{ylm} &&\lim_{\varepsilon\to0^+} \frac{1}{\varepsilon} \PP
\bigl[Y_{L + 10Y(t)}^-(t + \varepsilon) \geq(1 + c) Y_L^-(t) | A(t)
\bigr]
\nonumber
\\[-8pt]
\\[-8pt]
\nonumber
&&\qquad > c(\Delta) \frac{Y(t)}{Y(t)
+ X(t) + L}.
\end{eqnarray}
\end{lemma}

\begin{pf}
We will prove formula (\ref{ylp}). The proof of (\ref{ylm}) is
completely analogous, and the fact that (\ref{ygrowseq}) is true will
follow immediately from (\ref{ylp}) by taking $L=-X(t)$ and $\Delta=
1$.

Let $(x_0,y_0)$ be the point attaining the maximum $y_0=Y_L^+(t)$.
Denote $B= B_H((x_0,y_0), 2)$ and $y_1 = \max\{y; \exists x \mbox{
s.t. }   (x,y) \in B \}$.

Fix a constant $c>0$, which will be the universal constant in (\ref
{ylp}), whose value will be chosen later. If there exists a point
$(x,y) \in A(t)$ such that $x \geq L-10Y(t)$ and $y \geq(1+c) y_0$
then the event in (\ref{ylp}) holds almost surely, and we are done.
Therefore, we may assume from this point on that this is not the case,
hence, we can assume from now on that
%
%e15 #&#
\begin{equation}
\label{atcontained} A(t) \cap\bigl[L - 10 Y(t), \infty\bigr) \times\bigl[(1+c) y_0,
\infty\bigr) = \varnothing.
\end{equation}
Define the set
\[
U_c = \bigl( \bigl[L - 10 Y(t), \infty\bigr) \times\bigl[0, (1+c)
y_0\bigr]\bigr ) \setminus B.
\]
It is easy to verify that $d_H((x_0, y_1), U_0) \geq2 + c_1$ for a
universal constant $c_1 \geq0$ [recall that $d_H((x_0, y_0), (x_0,
y_1)) = 2$ and see Figure~\ref{fige}]. Therefore, by continuity by the
invariance of the metric to rescaling around the point $(x_0,0)$ (which
follows from Fact \ref{invariance}), we can choose the constant $c>0$
to be a small enough universal constant so that
%
%e16 #&#
\begin{equation}
\label{equt} d_H\bigl((x_0, y_1),
U_c\bigr) \geq2 + c_2
\end{equation}
for some universal constant $c_2 > 0$. Define
\[
S = \bigl\{(x,y) \in\partial B; d_H\bigl((x,y),
(x_0,y_1)\bigr) \leq c_2 / 2 \bigr\}
\]
%f3 #&#
\begin{figure}

\includegraphics{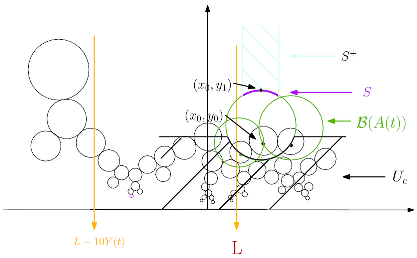}

\caption{Some of the definitions in the proof of Lemma \protect\ref{lemY}.} \label{fige}
\end{figure}
\hspace*{-3pt}(also see Figure~\ref{fige}). Equations (\ref{atcontained}) and (\ref
{equt}) imply that $S \subset\partial\BB(A(t))$. Note that this fact
does cease to be true if we make the constants $c,c_2$ smaller.
Therefore, by decreasing the value of these constants if necessary, we
can also assert that
%
%e17 #&#
\begin{equation}
\label{shigh} S \subset[L - 10 y_0, \infty) \times\bigl[(1+c)
y_0, \infty\bigr)
\end{equation}
(here we also used the fact that $x_0 \geq L$). Thanks to the last
equation and in view of equations (\ref{newparticle}) and (\ref
{newparticleout}), we have
\begin{eqnarray*}
&&\lim_{\varepsilon\to0^+} \frac{1}{\varepsilon} \PP\bigl[Y_{L - 10y_0}^+(t +
\varepsilon) \geq(1 + c) y_0 | A(t) \bigr] \\
&&\qquad=
\MM_{\BB(A(t))} \bigl( \bigl\{ (x,y) \in\partial\BB\bigl(A(t)\bigr); y \geq(1
+ c) y_0 \mbox{ and } x \geq L - 10y_0 \bigr\} \bigr)
\\
&&\qquad
\geq
\MM_{\BB(A(t))} (S).
\end{eqnarray*}
It is therefore enough to prove that
%
%e18 #&#
\begin{equation}
\label{mbneedtoshow} \MM_{\BB(A(t))} (S) \geq c(\Delta) \frac{Y(t)}{Y(t) + X(t) - L}.
\end{equation}
Our next goal thus to give a lower bound for $\MM_{\BB(A(t))}(S)$. We
do this in three steps.

\textit{Step} 1:
Define the set
\[
F = [x_0, \infty) \times( 2 \Delta y_1, \infty).
\]
In this step, we aim at showing that
%
%e19 #&#
\begin{equation}
\label{coolfact} \MM_{\BB( A(t)), F} (S) \geq c(\Delta)
\end{equation}
for some $c(\Delta) > 0$ which is a constant only depending on $\Delta
$. Define
\[
S^+:= \bigcup_{(x,y) \in S } \{x\} \times(y, \infty)
\]
and
\[
E = \PH\setminus S^+.
\]
Assumptions (\ref{atcontained}) and (\ref{equt}) along with Lemma
\ref
{ballslem} ensure that
\[
S^+ \cap\BB\bigl(A(t)\bigr) = \varnothing
\]
(see Figure~\ref{fige}) which implies that
%
%e20 #&#
\begin{equation}
\label{eqES} \MM_{\BB( A(t)), F} (S) \geq\MM_{E, F} (S).
\end{equation}
In order to give a bound for the right-hand side, we consider the transformation
\[
T\dvtx (x,y) \to\bigl((x - x_0) / y_0, y /
y_0 \bigr).
\]
By Fact \ref{invariance}, we know that $T$ is an isometry. Now, it is
not hard to verify that the sets $T(E)$ and $T(F)$ do not actually
depend on the aggregate $A(t)$, they only depend on the constant
$\Delta
$. It follows that there exists some constant $c(\Delta)$ such that
\[
\MM_{E, F} (S) = \MM_{T(E), T(F)} (S) = c(\Delta) > 0.
\]
It is also easy to verify (by drawing a picture) that $c(\Delta) > 0$
for all $\Delta\geq1$. By combining this with (\ref{eqES}), equation
(\ref{coolfact}) is proven.

\textit{Step} 2: Define
\[
G = (L, \infty) \times\bigl(2 \bigl(\Delta y_1 + X(t) - L\bigr),
\infty\bigr)
\]
and
\[
f(x,y) = \PP \biggl({ {\mbox{ Brownian motion started at } (x,y) } \atop {\mbox{
reaches }G \mbox{ before reaching }\BB\bigl( A(t)\bigr) }} \biggr).
\]
The aim of this step is to estimate $\inf_{(x,y) \in F} f(x,y)$. Along
with the previous step, this will give us a bound for $\MM_{\BB
(A(t)),G}(S)$.

In order to do this, we use the fact that $y$ coordinate of the
Brownian motion is a martingale whose starting value is at least $2
\Delta y_1$, together with the optional stopping theorem, to deduce
that the $y$ coordinate of the Brownian motion hits the set $2 (\Delta
y_1 + X(t) - L)$ before hitting the set $[0, \Delta y_1]$ with
probability at least $p':= \frac{\Delta y_1}{2 (\Delta y_1 + X(t) -
L)}$. Now since, by definition, $x_0 \geq L$, it follows from the
symmetry of the $x$ coordinate of the Brownian motion and from the
independence between the two coordinates that
%
%e21 #&#
\begin{equation}
\label{eqstep2} \inf_{(x,y) \in F} f(x,y) \geq\frac{\Delta y_1}{4 (\Delta y_1 + X(t) -
L)} \geq
c' \frac{Y(t)}{Y(t) + X(t) - L},
\end{equation}
where $c' > 0$ is a universal constant.

\textit{Step} 3: In view of that last step, it is enough to estimate
the probability that a Brownian motion starting from any point in $G$
will hit the set $\HH(\infty)$ before hitting $\BB(A(t))$. To show
that, we define
\[
H = \bigl(-\infty, \Delta y_1 + X(t)\bigr) \times\bigl[0, \Delta
y_1 + X(t) - L\bigr].
\]
Note that $A(t) \subset H$, so it is enough to estimate the probability
of reaching $\PH(\infty)$ before hitting $H$. The key
in this step is to define
\[
T\dvtx (x,y) \to\bigl((x - L) / \bigl(\Delta y_1 + X(t) - L\bigr),
y / \bigl(\Delta y_1 + X(t) - L\bigr) \bigr).
\]
Again, by Fact \ref{invariance}, we know that $T$ is an isometry. Moreover,
\[
T(G) = [0, \infty) \times[2, \infty),\qquad  T(H) = (-\infty, 1] \times[0,1].
\]
Viewed this way, it is clear that thanks to the conformal invariance
there exists a universal constant $c_3 > 0$ such that the probability
of a Brownian motion starting from any point in $G$ to hit to $x$ axis
before hitting $H$ is greater than $c_3$. Plugging this fact together
with (\ref{coolfact}) and (\ref{eqstep2}) finally gives
%
%e22 #&#
\begin{equation}
\MM_{\BB(A(t))} (S) \geq c_3 c(\Delta) c'
\frac{Y(t)}{Y(t) + X(t) - L},
\end{equation}
which is exactly (\ref{mbneedtoshow}), and the proof is complete.
\end{pf}

%re3.4 #&#
\begin{remark}
It is not hard to verify that the above proof gives us a rather poor
dependence of the constant $c(\Delta)$ on $\Delta$, namely, $c(\Delta)
\sim\exp(- \Delta^2)$. However, it is possible to prove that, in fact,
one can have the dependence $c(\Delta) \sim\Delta^{-1}$. Since this
difference will only affect the magnitude of the universal constant we
get in our main theorem, we choose to only present the above proof,
which is simpler.
\end{remark}

Finally, we will need the following lemma which will allow us to use
the optional stopping theorem.

%le3.5 #&#
\begin{lemma} \label{lemstopping}
Fix an aggregate $A(t)$ at time $t$, and fix a number $x_0>0$. Define
the stopping time,
\[
T = \min\bigl\{s \geq t; X(s) > x_0 | A(t) \bigr\}.
\]
Then
\[
\EE[T] \leq\infty.
\]
\end{lemma}

The proof is not hard but rather technical, and we only provide a
sketch. One way to explain the reason behind this fact is that the
equilibrium measure on a geodesic line in the hyperbolic plane exists,
and is a constant multiple of the length measure. As a result, it
follows that the convex hull of the aggregate encapsulates any ball
within a time whose expectation is finite.

\begin{pf*}{Proof of Lemma \ref{lemstopping} (\normalfont{Sketch})}
Consider the domain
\[
L = \bigl\{(x,y) \in\PH; x^2 + y^2 > 1 \bigr\}.
\]
It is well known that for any two geodesic curves, there exists an
isometry of the hyperbolic plane sending the first to the second.
Consequently, there is a bijective isometry $T$ such that
\[
T \bigl( \bigl\{(x,y) \in\PH; x \geq x_0 \bigr\} \bigr) = L.
\]
Therefore, by considering the initial aggregate $T(A(t))$, without loss
of generality we may assume that
\[
T' = \min\bigl\{t \geq0; A(t) \cap L \neq\varnothing | A(0) \bigr
\}
\]
and prove that $\EE[T'] \leq\infty$ for an arbitrary initial aggregate
$A(0)$. Defining,
\[
T_1 = \min\bigl\{t; X(t) \geq1 \mbox{ or } Y(t) \geq1 \bigr\}.
\]
It is clear that $T' \leq T_1$, therefore, it is enough to show that
$\EE[T_1] < \infty$. Lemma \ref{lemY} teaches us that for any $t
\leq T_1$
one has
\[
\lim_{\varepsilon\to0^+} \frac{1}{\varepsilon} \PP\bigl( Y(t + \varepsilon) \geq
(1+c) Y(t) | \FF_t \bigr) \geq c Y(0)
\]
for a universal constant $c>0$. It is not hard to check that the last
equation implies that there exists a constant $c_1$ which only depends
on $Y(0)$ such that
\[
\PP\bigl(Y(t+1) \geq1 \mbox{ or } X(t+1) \geq1 | \FF_t\bigr) \geq
c_1
\]
for all $t > 0$. In other words,
\[
\PP(T_1 < t + 1 | \FF_t) \geq c_1\qquad
\forall t \geq0.
\]
The above equation implies that $T_1$ has a subexponential tail and,
therefore, has a finite expectation.
\end{pf*}

\begin{pf*}{Proof of Lemma \ref{lemannoying}}
Fix $t \geq0$ and fix an aggregate $A(t)$. Define $n_0 = i(t) = \#
A(t) + 1$. We begin with noting that Lemma \ref{ballslem} teaches us that
\[
A_{n_0 + n} \subset\RR\times\bigl[0, Y(t) 10^n\bigr]\qquad \forall n
\geq1
\]
and, therefore,
%
%e23 #&#
\begin{equation}
\label{eqxgrowth} X(t_{n_0 + n}) \leq X(t) + 7 Y(t) 10^n \qquad\forall
n \geq1
\end{equation}
almost surely. We claim that, in order to conclude the lemma, it will
be enough to show that there exist constants $C', \eps_0>0$ [which may
depend on $A(t)$] such that
%
%e24 #&#
\begin{equation}
\label{ntsannoying} \PP\bigl(t_{n_0 + n} - t < \eps| A(t)\bigr) <
C' \eps10^{-2n}\qquad \forall\eps < \eps_0.
\end{equation}
Indeed, for all $\alpha> 0$, write
\[
n = \max \biggl( \biggl\lfloor\frac{\log(\alpha/ 7 Y(t))}{\log10} \biggr\rfloor, 1 \biggr).
\]
Then thanks to (\ref{eqxgrowth}),
\[
\PP \bigl(X(t + \eps) - X(t) > \alpha\mid A(t) \bigr) \leq\PP\bigl(
t_{n_0 + n} \leq t + \varepsilon| A(t)\bigr)
\]
and plugging (\ref{ntsannoying}) to this would prove the lemma.

We therefore move on to the proof of (\ref{ntsannoying}). Recall that
for all $j$, the difference $t_j - t_{j-1}$ is an
exponentially-distributed random variable whose expectation is
$\operatorname{Cap}(A_j)^{-1}$. Moreover, we clearly have
by the definition of the harmonic measure
\begin{eqnarray*}
\operatorname{Cap}(A_j)& =& \MM_{\BB(A_j)} \bigl(\partial\BB(A_j)
\bigr) = \sum_{p \in A_j} \MM _{\BB(A_j)} \bigl(\partial
\BB\bigl(\{p\}\bigr)\bigr) \\
&\leq&
\sum_{p \in A_j} \MM_{\BB(\{p\})} \bigl(\partial\BB
\bigl(\{p\}\bigr)\bigr) = (j + 1) C_0\qquad \forall j \geq0,
\end{eqnarray*}
where $C_0>0$ is some universal constant. It follows that for all $j <
n$, the expectation
of $t_{j+1} - t_{j}$ is at least $\frac{1}{n C_0}$. An elementary fact
about exponentially-distributed variables is that
\[
0 < a < b \quad\Rightarrow\quad\PP\bigl(E[b] < t\bigr) < \PP\bigl(U\bigl([0,a]\bigr) < t
\bigr) \qquad\forall t > 0,
\]
where $U([0,a])$ represents a uniformly-distributed point in the
interval $[0,a]$. It follows that
\[
\PP(t_{n_0 + n} - t_{n_0 + 1} < \eps) \leq\PP \Biggl( \sum
_{i=1}^{n-1} X_i < \eps \Biggr)\qquad \forall
\eps>0, \forall n \geq1,
\]
where $X_i$ are independent variables whose distribution is uniform
over the interval $ [0, \frac{1}{C_0(n_0+n)}  ]$. An application
of a standard large-deviation principle teaches us that there exists
some $\eps_0 > 0$ (which may depend on $n_0$) such that
\[
\PP\bigl(t_{n_0 + n} - t_{n_0 + 1} < \eps | A(t_{n_0+1})\bigr)
\leq10^{-2n}
\]
for all $\eps< \eps_0$ and for all $n > 1$.
Moreover, since the density of the exponential distribution is bounded,
we have
\[
\PP\bigl(t_{n_0 + 1} - t < \eps| A(t)\bigr) \leq C_2 \eps\qquad
\forall\eps> 0
\]
for some constant $C_2$. Plugging the two above estimates finally
establishes equation (\ref{ntsannoying}) and the lemma is complete.
\end{pf*}

%s4 #&#
\section{The process of ratios}\label{sec4}
For all $t \geq0$, define $R(t) = X(t) / Y(t)$. The goal of this
section is to prove the following theorem.

%th4.1 #&#
\begin{theorem} \label{mainsec4}
There exists a universal constant $C>0$ such that the following holds:

Let $t \geq0$ be a time and fix any initial configuration $A(t)$. In
addition, fix a number $X_0$ such that $X_0 \geq X(t)$. Define, for
every nonnegative integer $i$,
\[
\tau_i = \min\bigl\{s; X(s) \geq2^i X_0
\bigr\}.
\]
Then one has for all $i$,
\[
\EE\bigl[R(\tau_{i+1}) | A(t)\bigr] \leq C + 0.9 \EE\bigl[R(
\tau_i) | A(t)\bigr].
\]
\end{theorem}

The next lemma, which is one of the two main ingredients in the proof
of the theorem, gives upper bounds on the expected growth of $R(t)$.
Its proof relies on a combination of Lemmas \ref{lemX} and \ref{lemY}.

%le4.2 #&#
\begin{lemma} \label{EtLem}
There exist universal constants $\delta, c_1, c_2>0$ such that one has
for all $t \geq0$,
%
%e25 #&#
\begin{equation}
\label{eqetlem1} \limsup_{\varepsilon\to0^+} \frac{1}{\varepsilon} \EE\bigl[R(t+
\varepsilon) - R(t) | A(t) \bigr] < + c_1.
\end{equation}
%%
%\label{eqetlem2}
Moreover, defining the following event,
%
%e26 #&#
\begin{equation}
\label{defEt} E(t):= \bigl\{ \tilde Y(t) < \delta Y(t) \bigr\},
\end{equation}
whenever the event $E(t)$ holds one has
%
%e27 #&#
\begin{equation}
\limsup_{\varepsilon\to0^+} \frac{1}{\varepsilon} \EE\bigl[R(t+ \varepsilon) - R(t)
| A(t) \bigr] < - c_2.
\end{equation}
\end{lemma}

\begin{pf}
Denote
\[
F(\varepsilon) = \bigl\{Y(t + \varepsilon) \geq(1+c) Y(t) \bigr\},
\]
where $c$ is the constant from equation (\ref{ygrowseq}). According to
Lemma \ref{lemY}, we have
%
%e28 #&#
\begin{equation}
\label{lemy1eq} \PP\bigl(F(\varepsilon) | A(t)\bigr) \geq c \varepsilon
\frac{Y(t)}{X(t) + Y(t)} + o(\varepsilon).
\end{equation}
%
%where we have used equation (\ref{Xbig}) in order to replace $X(t) +
%Y(t)$ by $X(t)$.
Next, we use Lemma \ref{lemX} to deduce that
%
%e29 #&#
\begin{equation}
\label{lemx1eq} \EE\bigl[X(t + \varepsilon) - X(t) | A(t) \bigr] \leq C \tilde Y(t)
\varepsilon+ o(\varepsilon).
\end{equation}
We write
\begin{eqnarray*}
&&\EE\bigl[X(t+ \varepsilon) / Y(t + \varepsilon) | A(t) \bigr]\\
&&\qquad =
\EE \biggl[  \dfrac{X(t+ \varepsilon)}{ Y(t + \varepsilon)}
\mathbf {1}_{F(\varepsilon)^C} \Big\vert
\mathcal{F}_t \biggr] + \EE \biggl[ \dfrac{X(t+ \varepsilon)}{ Y(t + \varepsilon)} \mathbf
{1}_{F(\varepsilon
)} \Big\vert \mathcal{F}_t \biggr]\\
&&\qquad \leq
\frac{1}{Y(t)} \EE\bigl[X(t+ \varepsilon) \mathbf{1}_{F(\varepsilon)^C} | \mathcal
{F}_t \bigr] + \frac{1}{(1+c) Y(t)} \EE\bigl[X(t+ \varepsilon) \mathbf
{1}_{F(\varepsilon)} | \mathcal{F}_t \bigr]\\
&&\qquad =
\frac{1}{Y(t)} \EE\bigl[X(t+ \varepsilon) | \mathcal{F}_t \bigr] -
\biggl( 1 - \frac
{1}{1+c} \biggr) \frac{1}{Y(t)} \EE\bigl[X(t+ \varepsilon)
\mathbf {1}_{F(\varepsilon)} | \mathcal{F}_t \bigr]\\
&&\qquad =
\frac{X(t)}{Y(t)} + \frac{1}{Y(t)} \EE\bigl[X(t+ \varepsilon) - X(t) | \mathcal
{F}_t \bigr] - \frac{c}{1+c} \frac{1}{Y(t)} \EE\bigl[X(t+
\varepsilon) \mathbf {1}_{F(\varepsilon)} | \mathcal{F}_t \bigr]\\
&&\qquad \leq
\frac{X(t)}{Y(t)} + \frac{1}{Y(t)} \EE\bigl[X(t+ \varepsilon) - X(t) | \mathcal
{F}_t \bigr] - \frac{c}{1+c} \frac{X(t)}{Y(t)} \PP\bigl(F(
\varepsilon) | \mathcal {F}_t \bigr).
\end{eqnarray*}
Plugging equations (\ref{lemy1eq}) and (\ref{lemx1eq}) into this
formula gives
\[
\EE\bigl[X(t+ \varepsilon) / Y(t + \varepsilon) | \mathcal{F}_t \bigr]
\leq\frac
{X(t)}{Y(t)} + C \frac{\tilde Y(t)}{Y(t)} \varepsilon- c_3 \varepsilon
\frac
{X(t)}{X(t) + Y(t)} + o(\varepsilon)
\]
for a universal constant $c_3>0$. Since $\tilde Y(t) \leq Y(t)$ by
definition, equation (\ref{eqetlem1}) follows. To prove the second
part of the lemma, the reader may easily verify that by the definition
of the event $E(t)$, whenever $E(t)$ holds with $\delta< 1$, one has
%
%e30 #&#
\begin{equation}
\label{Xbig} X(t) > c_4 Y(t)
\end{equation}
for a universal constant $c_4 > 0$. Moreover, by definition of the
event $E(t)$ one has
\[
\frac{\tilde Y(t)}{Y(t)} \leq\delta.
\]
Plugging in these two facts gives
\[
\EE\bigl[X(t+ \varepsilon) / Y(t + \varepsilon) | \mathcal{F}_t \bigr]
\leq\frac
{X(t)}{Y(t)} + \varepsilon \biggl(C \delta- c_3
\frac{1}{1 + c_4^{-1}} \biggr) + o(\varepsilon).
\]
Thus, by choosing $\delta$ to be a small enough universal constant, the
second part of the lemma is also established.
\end{pf}

As a corollary, we get
the following.

%co4.3 #&#
\begin{corollary} \label{corEt}
There is a universal $\delta> 0$ such that if we define the event
$E(t)$ as in (\ref{defEt}), then the following holds:
suppose $A(t)$ is such that $E(t)$ holds. Define
\[
T = \min\bigl\{s>t; E(s) \mbox{ does not hold or } X(s) > 1.1 X(t) \mbox{ or }
Y(s) > 100 Y(t) \bigr\}.
\]
Then one has
\[
\PP\bigl(X(T) \geq1.1 X(t) | \FF_t \bigr) \leq0.01.
\]
\end{corollary}

\begin{pf}
Using the optional stopping theorem (which is justified thanks to Lemma
\ref{lemstopping}) with the result of the previous lemma, we have for a
small enough choice of $\delta$,
\[
\EE\bigl[ X(T) / Y(T) | \FF_t\bigr] \leq X(t) / Y(t) -
c_2 \EE[T - t].
\]
Since the left-hand side cannot be negative,
\[
\EE[T-t] \leq\frac{1}{c_2} \frac{X(t)}{Y(t)}.
\]
According to Lemma \ref{lemX}, the following process is a super-martingale:
%
%e31 #&#
\begin{equation}
s \to X(t+s) - X(t) - \int_t^{t+s} C \tilde Y(r)
\,dr.
\end{equation}
Therefore, by the optional stopping theorem, and since for every $t
\leq s < T$ we have by definition $\tilde Y(s) \leq\delta Y(s) \leq
100 \delta Y(t)$,
%
%e32 #&#
\begin{eqnarray}
\EE\bigl[X(T) - X(t) | \FF_t \bigr]& \leq& C \EE \biggl[\int
_t^T \tilde Y(s) \,ds \biggr] \leq 100 C \delta
Y(t) \EE[T-t]
\nonumber
\\[-8pt]
\\[-8pt]
\nonumber
& \leq&
C' \delta Y(t) \frac{X(t)}{Y(t)} \leq C' \delta
X(t).
\end{eqnarray}
Again, by choosing $\delta$ small enough (note that it can always be
made smaller without affecting the result of the previous lemma), we
can make sure that
\[
\EE\bigl[X(T) - X(t) | \FF_t \bigr] \leq0.001 X(t),
\]
and since $X(t)$ is increasing it follows by Markov's inequality that
\[
\PP\bigl(X(T) \geq1.1 X(t) | \FF_t \bigr) \leq0.01,
\]
which is the promised result.
\end{pf}

From this point on, we assume that the event $E(t)$ is defined as in
equation (\ref{defEt}), and the constant $\delta$ is a fixed positive
universal constant taken to be small enough such that the above
corollary holds true.

In view of the above corollary, the only times we have to worry about
are whenever $E(t)$ does not hold. The next lemma in some sense
complements the previous one, ensuring us that also if $E(t)$ does not
hold, we should expect $X(t) / Y(t)$ to decrease after a while (due to
completely different reasons), providing that it is not too small.

%le4.4 #&#
\begin{lemma} \label{noEtLem}
There exists a universal constant $\Gamma>0$ such that the following holds:
Assume that for some $t_0 \geq0$, $E(t_0)$ does not hold and $X(t_0) /\break
Y(t_0) > \Gamma$, then
\[
\EE\bigl[X(t_1) / Y(t_1) | A(t_0) \bigr]
< \tfrac{1}{4} X(t) / Y(t),
\]
where $t_1 = \min\{ s;   X(s) \geq1.1 X(t_0) \}$.
\end{lemma}

Before we move on to the proof, let us try to explain why this bound
should be correct. Whenever the event $E(t)$ does not hold, we know
that there is a particle $p$ located close to the front of the
aggregate which, up to a constant, attains the vertical height of the
entire aggregate, $Y(t)$. In this case, we can effectively ``restart''
the growth process by only considering the part close to the front of
the aggregate, while ignoring the rest of it: as a consequence of Lemma
\ref{lemY}, we know that parts of the aggregate located close to the
front have a vertical growth rate which is proportional only to the
distance from the front. This means that when considering only the
latter part of the aggregate, the growth rate will no longer be a
function of $X(t)$. Now, as a result of Lemma \ref{ballslem}, the
vertical growth of the particles is multiplicative in the sense that in
order for $Y(t)$ to multiply itself by a constant, it is enough for the
particle $p$ to duplicate itself upward a constant number of times.
From this point on, the proof relies on a compactness-type argument: we
know that the top particle has to duplicate a constant number of times,
while the rate of duplication is independent of $X(t)$. Therefore, it
is enough to establish that the universal rate of growth is such that
any number of duplications will occur eventually, with high
probability. The time that it takes, which affects the increment of
$X(t)$, can then be absorbed into the constant $\Gamma$; When this
constant is big enough, a prescribed additive growth of $X(t)$ results
in a small multiplicative growth which does not significantly affect $R(t)$.

The proof will be divided into a few steps. In the first step, we
demonstrate that it suffices to show that there exists a constant $C>0$
such that $Y(t)$ multiplies itself by some constant, say 5, before
$X(t)$ grows (additively) by $C$. The second and third steps deal with
the rate of duplications of the particle $p$ mentioned above. It is
shown that within any time interval in which $X(t) - X(t_0)$ multiplies
itself by two, there is at least a constant probability for the
particle $p$ to duplicate itself once. This is the ``compactness'' to
which we were referring above, as this rate does not depend on
$X(t_0)$. In the fourth and last step, we iteratively use this fact to
conclude that there is a probability bounded from below for any
constant number of multiplications when the time interval is large
enough.

\begin{pf*}{Proof of Lemma \ref{noEtLem}}
Since the claim is invariant to rescaling around the origin, we may
assume that $Y(t_0) = 1$. Define
\[
T = \min\bigl\{s; Y(s) > 5 \bigr\}.
\]
\textit{Step} 1:
We claim that it is enough to show that there exists a universal
constant $C>0$ such that
%
%e33 #&#
\begin{equation}
\label{mainstepeq} \PP\bigl(X(T) < X(t_0) + C\bigr) > 0.99.
\end{equation}
Let us explain why this fact suffices in order to complete the proof.
Since almost surely only one particle can be added at a time and
assuming that $\Gamma$ is a large enough constant, an application of
Lemma \ref{ballslem} gives
\[
X(t_1) \leq1.1 X(t_0) + 10 Y(t_0)
\leq1.11 X(t_0).
\]
Also, if $\Gamma$ is large enough then we can assume that $ X(t_0) + C
< 1.1 X(t_0)$
which implies that
\[
\PP(T < t_1) > 0.99.
\]
Using these two facts, we can thus estimate
\begin{eqnarray*}
&&\EE\bigl[X(t_1) / Y(t_1) | \FF_{t_0}\bigr] \\
&&\qquad=
\EE\bigl[X(t_1) / Y(t_1) \mathbf{1}_{\{T < t_1\}} |
\FF_{t_0}\bigr] + \EE\bigl[X(t_1) / Y(t_1)
\mathbf{1}_{\{t_1 \leq T\}} | \FF_{t_0}\bigr] \\
&&\qquad\leq
\EE\bigl[1.11 X(t_0) / \bigl(5 Y(t_0)\bigr)
\mathbf{1}_{\{T < t_1\}} | \FF_{t_0}\bigr] + \EE \bigl[1.11
X(t_0) / Y(t_0) \mathbf{1}_{\{t_1 \leq T\}} |
\FF_{t_0}\bigr] \\
&&\qquad\leq
\bigl(\tfrac{1.11}{5} + 1.11 \cdot0.01 \bigr) X(t_0) /
Y(t_0) < \tfrac
{1}{4} X(t_0) /
Y(t_0),
\end{eqnarray*}
which is the result.

\textit{Step} 2:
Define $Z(s) = X(s) - X(t_0) + 5$. According to the assumption that
$E(t_0)$ does not hold and by definition of $\FR(A(t))$, we know that
either $Y_{X(t_0) - 5}^+(t_0)$ or $Y_{-X(t_0) + 5}^-(t_0)$ are greater
than the universal constant $\delta>0$. Assume without loss of
generality that
%
%e34 #&#
\begin{equation}
\label{ylbig} Y_{X(s) - 5}^+(t_0) \geq\delta
\end{equation}
(the assumption is legitimate since the model is invariant under
reflection around the $y$ axis).
Define $\Delta= 100 \delta^{-1}$. The assumption \eqref{ylbig},
together with the definitions of $Y_L^+(s)$ and $T$, implies that for
any $L < X(t_0) - 5$ and for any $t_0 \leq s \leq T$ one has $Y(s) \leq
\Delta Y_L^+(s)$. Therefore, we can use the second part of Lemma \ref
{lemY} to deduce that there exists a universal constant $c_1>0$ such
that for all $t_0 \leq s < T$ and for all $L < X(t_0) - 5$ one has
%
%e35 #&#
\begin{equation}
\label{ygrows} \lim_{\varepsilon\to0^+} \frac{1}{\varepsilon} \PP
\bigl[Y_{L - 50}(s + \varepsilon) > (1 + c_1)
Y_{L}(s) | A(s) \bigr] > c_1 / \bigl(5 + X(s) - L\bigr).
\end{equation}
Here, we used the assumption that for $s < T$, one has $Y(s) \leq5$.

Define
\[
L_0 = 5 + 50 \log_{(1+c_1)} \Delta.
\]
At this point, the reader may regard $L_0$ as some large universal
constant, its significance will become clear later on. Let $L$ be a
number satisfying
%
%e36 #&#
\begin{equation}
\label{condL} X(t_0) - L_0 \leq L \leq
X(t_0) - 5.
\end{equation}
Also, fix a time $t_0 \leq t < T$ and define
\[
T_1 = \min\bigl\{s| Z(s) > 2 Z(t) \bigr\}.
\]
Let $N(s)$ be a random variable counting the number of ``jumps'' up to
time $s$, hence,
\[
N(s) = \# \Bigl\{ r \in[t,s]; Y_{L-50}^+(r) \geq(1 + c_1)
\lim_{\varepsilon\to0^+} Y_L^+(r - \varepsilon) \Bigr\}.
\]
Our next goal will be to show that there exists a universal constant
$c>0$ such that
%
%e37 #&#
\begin{equation}
\label{finaltau} \PP\bigl(T < T_1 \mbox{ or } N(T_1)
\geq1 | \FF_t \bigr) \geq c,
\end{equation}
which will be done in the next step.

\textit{Step} 3:
To prove the last formula, we begin by defining
\[
M(s) = N(s) - c_1 (s-t) / \bigl(2 Z(t) + L_0\bigr).
\]
By equation (\ref{ygrows}) and by the fact that $L \geq X(t_0) - L_0$,
we learn that $M(s)$ is a sub-martingale in the interval $[t, T_1
\wedge T]$. Thus, by the optional stopping theorem (which we can use
thanks to Lemma \ref{lemstopping}), one has
\[
\EE\bigl[ N(T_1 \wedge\tau) | \FF_t \bigr] \geq\EE
\bigl[(T_1 \wedge\tau- t ) | \FF _t \bigr]
c_1 / \bigl(2Z(t) + L_0\bigr),
\]
where $\tau= \min\{t| N(t) \geq1 \} \wedge T$. Consequently, for all
$\alpha> 0$, we may calculate
\begin{eqnarray*}
&&\PP(\tau< T_1 | \FF_t)\\
&&\qquad \geq\PP\bigl(
N(T_1) \geq1 | \FF_t \bigr) \geq\EE
\bigl[N(T_1 \wedge\tau) | \FF_t\bigr] \\
&&\qquad\geq
\EE\bigl[(T_1 \wedge\tau- t ) | \FF_t\bigr]
c_1 / \bigl(2Z(t) + L_0\bigr) \\
&&\qquad\geq
\EE\bigl[(T_1 - t) \mathbf{1}_{ \{\tau> T_1\} } | \FF_t
\bigr] c_1 / \bigl(2 Z(t) + L_0\bigr) \\
&&\qquad\geq
\bigl(\PP\bigl(T_1 - t > 2 \alpha Z(t) | \FF_t \bigr)
- \PP(T_1 > \tau| \FF_t)\bigr) 2 \alpha Z(t)
c_1 / \bigl(2 Z(t) + L_0\bigr)
\end{eqnarray*}
[using the assumption $Z(t) \geq5$]
\[
\geq\bigl(\PP\bigl(T_1 - t > \alpha2Z(t) | \FF_t\bigr) -
\PP(T_1 > \tau| \FF_t)\bigr) \alpha c_2,
\]
for some universal constant $c_2 > 0$.
Thus,
%
%e38 #&#
\begin{equation}
\label{eqzt} \PP(\tau< T_1| \FF_t) \geq\alpha
c_2 \PP\bigl(T_1 - t > 2 \alpha Z(t) | \FF
_t\bigr) / (1 + c_2 \alpha).
\end{equation}
We now use Lemma \ref{lemX}, combined with the fact that $Y(s) < 5$ for
all $t \leq s \leq T$, according to which
\[
\EE\bigl(Z(s \wedge T) - Z(t) | \FF_t \bigr) < C_1
(s-t)
\]
for a universal constant $C_1>0$. Taking $s=t + 2 \alpha Z(t)$ and
using Markov's inequality, we get
\[
\PP \bigl(Z \bigl(\bigl(t + 2 \alpha Z(t)\bigr) \wedge T \bigr) > 2 Z(t) \mid
\FF_t \bigr) < 2 C_1 \alpha.
\]
Now, by the definition of $T_1$,
\[
\bigl\{Z \bigl(\bigl(t + 2 \alpha Z(t)\bigr) \wedge T \bigr) < 2 Z(t) \bigr\}
\subseteq \bigl\{Z\bigl(t + 2 \alpha Z(t)\bigr) < 2 Z(t) \bigr\} \cup \{ T <
T_1 \}
\]
so a union bound gives
\[
\PP\bigl(Z\bigl(t + 2 \alpha Z(t)\bigr) < 2 Z(t) | \FF_t \bigr) > 1
- 2 C_1 \alpha- \PP(T < T_1 | \FF_t).
\]
But, using the definition of $T_1$ once more, we know that
\[
Z\bigl(t + 2 \alpha Z(t)\bigr) < 2 Z(t) \quad\Rightarrow\quad T_1 \geq t + 2
\alpha Z(t)o
\]
and the last equation becomes
\[
\PP\bigl(T_1 - t > 2 \alpha Z(t) | \FF_t \bigr) \geq1
- 2 C_1 \alpha- \PP(T < T_1 | \FF_t).
\]
Choosing $\alpha$ to be a small enough universal constant and plugging
the above into~(\ref{eqzt}) gives
\[
\PP(\tau< T_1| \FF_t) \geq\alpha c_3
\bigl(1 - 2 C_1 \alpha- \PP(T < T_1| \FF_t)
\bigr) \geq c_4\bigl(1 - \PP(\tau< T_1 |
\FF_t)\bigr),
\]
where $c_3, c_4$ are universal constants. In other words, we have that
%
%e39 #&#
\begin{equation}
\PP(\tau< T_1| \FF_t) \geq c
\end{equation}
and equation (\ref{finaltau}) is proven.

\textit{Step} 4:
At this point, the strategy we will use in order to prove (\ref
{mainstepeq}) is to repeat this argument again and again, for a
sequence of times $Q_i$, until we accumulate enough ``jumps'' so that $T$
is surely reached. Define
\[
Q_1 = \min\bigl\{s \geq t_0; Z(s) > 2 Z(0) \bigr\}
\]
and inductively,
\[
Q_{i+1} = \min\bigl\{s \geq t_0; Z(s) > 2
Z(Q_i) \bigr\}.
\]
Also define $I$ to be the largest integer $i$ such that $Q_i < T$. By
the definition of $T$ and by Lemma \ref{ballslem}, we know that the
(Euclidean) radius of any added ball is smaller than a constant, so we
can easily deduce the ``continuity'' in the following sense:
\[
Z(Q_{i+1}) < R Z(Q_i)\qquad \forall1 \leq i < I,
\]
where $R$ is a universal constant. It follows that
%
%e40 #&#
\begin{equation}
\label{zqsmall} Z(Q_i) < 5 R^i\qquad \forall1 \leq i < I.
\end{equation}
For all $i \in\mathbb{N}$ define $N_i$ to be the number of ``jumps'' so
far. In other words, define $N_0 = 0$
and (recursively)
\[
N_i = \# \Bigl\{ j < i; \exists r \in(Q_j,
Q_{j+1}] \mbox{ such that } Y_{L_j - 50}^+(r) \geq(1 +
c_1) \lim_{\varepsilon\to0^+} Y_{L_j}^+(r - \varepsilon)
\Bigr\},
\]
where
\[
L_j:= X(t_0) - 5 - 50 N_j.
\]
Define also $Y_i = Y_{L_i}^+(Q_i)$, $Z_i = Z(Q_i)$ and $\FF_i$ to be
the $\sigma$-algebra generated by $A(Q_i)$.
Observe that, by (\ref{ylbig}), the number of jumps needed in order to
reach $T$ is smaller than $\log_{(1+c_1)} \Delta$. So, by definition,
\[
N_I \leq\log_{(1+c_1)} \Delta,
\]
which implies, by the definition of $L_0$, that
\[
X(t_0) - L_0 \leq L_i \leq
X(t_0) - 5\qquad\forall i \leq I.
\]
The above equation asserts that (\ref{condL}) is fulfilled, so we may
use equation (\ref{finaltau}) which translates to
%
%e41 #&#
\begin{equation}
\PP\bigl( Y_{i+1} > (1+ c_1) Y_i |
\FF_i \bigr) > c \qquad\forall1 \leq i < I.
\end{equation}
An application of, say, Hoeffding's inequality gives
%
%e42 #&#
\begin{equation}
\label{hoeffding} \PP \bigl( Y_{k} < (1+c_1)^{{kc}/{2}}
Y_0 \mbox{ and } I > k \vert \FF_{t_0} \bigr) <
C_3 \exp(-c_3 k),
\end{equation}
where $C_3,c_3>0$ are universal constants. Define $\alpha= \frac{2}{c}
\log_{(1+c_1)} \Delta$. By (\ref{ylbig}), we know that
\[
(1+c_1)^{{kc}/{2}} Y_0 > 5\qquad \forall k > \alpha,
\]
which by definition means that
\[
I > k\quad \Rightarrow \quad Y_{k} \leq(1+c_1)^{{kc}/{2}}
Y_0\qquad \forall k > \alpha.
\]
Equation (\ref{hoeffding}) becomes
\[
\PP( I > k | \FF_{t_0} ) < C_3 \exp(-c_3 k)\qquad
\forall k > \alpha.
\]
Now choose $k$ large enough universal constant such that $k > \alpha$
and also the right-hand side of the above equation is smaller than
$0.01$ (this is possible since $\Delta$ and $c_1$ have been fixed as
universal constants, so $\alpha$ is a universal constant). We get
\[
\PP( I > k | \FF_{t_0} ) < 0.01
\]
and along with (\ref{zqsmall}) this yields
\[
\PP \bigl( Z(T) > 5 R^{k} | \FF_{t_0} \bigr) < 0.01
\]
equation (\ref{mainstepeq}) follows and the proof is complete.
\end{pf*}

The next proposition combines the results of the previous two lemmas
together into a unified bound on the behavior of the process $R(t) =
X(t) / Y(t)$.

%pr4.5 #&#
\begin{proposition} \label{sumitup}
There exists a universal constant $C >0$ such that the following holds:
Assume that for some $t_0 \geq0$, $X(t_0) / Y(t_0) > C$. Then
\[
\EE\bigl[X(T) / Y(T) | A(t_0) \bigr] < \tfrac{1} 2
X(t_0) / Y(t_0),
\]
where $T = \min\{ s;   X(s) \geq1.3 X(t_0)\}$.
\end{proposition}

Once we have established the above lemmas, the idea of the proof is
very simple: just split into two cases, determined by whether or not
there is a point in time at which $X(t)$ has not yet reached the value
$1.1 X(t_0)$ and the event $E(t)$ does not hold. If such a point
exists, we use Lemma \ref{noEtLem}, otherwise, we use Corollary \ref{corEt}.

\begin{pf*}{Proof of Proposition \ref{sumitup}}
If the event $E(t_0)$ does not hold, just use Lemma \ref{noEtLem} with
the legitimate assumption that $C > \Gamma$ and we are done. Otherwise, denote
\[
T_1 = \min\bigl\{s>t_0; E(s) \mbox{ does not hold or }
X(s) > 1.1 X(t_0) \mbox{ or } Y(s) > 100 Y(t_0) \bigr\}
\]
and,
\[
T_2 = \min\bigl\{s>T_1; X(s) > 1.1 X(T_1)
\mbox{ or } Y(s) > 100 Y(t_0) \bigr\}.
\]
Using Lemma \ref{ballslem} and since we are stopping before $Y(s)$ has
reached the height $100 Y(t_0) \leq100 C^{-1} X(t_0)$,
we see that by taking the constant $C$ to be large enough, we can make
sure that any particle added to the aggregate before time $T_2$ can
increase $X(t)$ by no more than $0.01 X(t_0)$. Since almost surely only
one particle can be added at a time, and assuming that $C$ is a large
enough constant, we get
%
%e43 #&#
\begin{equation}
\label{xincrease} X(T_2) \leq1.3 X(T_0).
\end{equation}
Denote by $F$ the event that $E(T_1)$ holds. By Corollary \ref{corEt},
we know that
\[
\PP\bigl(X(T_1) \geq1.1 X(t) | \FF_t \bigr) \leq0.01.
\]
We can estimate
%
%e44 #&#
\begin{eqnarray}
\label{est1}&& \EE\bigl[X(T_2) / Y(T_2)
\mathbf{1}_{F} | \FF_{t_0} \bigr]
\nonumber
\\[-8pt]
\\[-8pt]
\nonumber
&&\qquad\leq
0.01 \cdot1.3 X(t_0) / Y(t_0) + \tfrac{1}{100} 1.3
X(t_0) / Y(t_0),
\end{eqnarray}
where we have used that fact that by definition of $T_1$ whenever
$X(t_0) < 1.1$ and $F$ holds, then necessarily $Y(T_1) > 100 Y(t_0)$.

Next, we handle the case that $F$ does not hold. By assuming that $C$
is large enough, we can assume that $X(T_1) / Y(T_1) > \Gamma$ (the
universal constant in the formulation of Lemma \ref{noEtLem}). An
application of Lemma \ref{noEtLem} gives
\[
\EE\bigl[X(T_2) / Y(T_2) \mathbf{1}_{F^C} |
\FF_{t_0} \bigr] \leq\tfrac{1}{4} \EE \bigl[ X(T_1) /
Y(T_1) \mathbf{1}_{F^C} | \FF_{t_0} \bigr] \leq
\tfrac{1.3}{4} X(t_0) / Y(t_0).
\]
Combining this bound with (\ref{est1}) gives us
\begin{eqnarray*}
&&\EE\bigl[X(T_2) / Y(T_2) | \FF_{t_0} \bigr]\\
&&\qquad
\leq
\bigl(0.013 + \tfrac{1.3}{100} + \tfrac{1.3}{4} \bigr)
X(t_0) / Y(t_0) \leq0.36 X(t_0) /
Y(t_0),
\end{eqnarray*}
by the same argument as the one preceding (\ref{xincrease}), one has
$X(T) < 1.2 X(T_2)$, which gives us the desired result.
\end{pf*}

We are finally in a position to prove the main theorem of this section.

\begin{pf*}{Proof of Theorem \ref{mainsec4}}
Define
\[
E_i = \bigl\{X(\tau_i) / Y(\tau_i) > C \bigr
\}
\]
and $E_0 = \{X(t) / Y(t) > C \}$, where $C$ is a universal constant
whose value will be determined later on. Observe that if for some $i$,
we have $X(\tau_i) > 1.01 \times2^i X_0$, it means that the last jump
in $X(t)$ must have been rather big, namely that for the smallest
integer $j$ such that $X(\tau_j) = X(\tau_i)$, one has
\[
X(\tau_j) > 1.01 \lim_{\varepsilon\to0^+ } X(\tau_j
- \varepsilon)
\]
[here we used the assumption that $X_0 \geq X(t)$]. This, in turn,
means that the radius of the last ball added was proportional to
$X(\tau
_i)$. By Lemma \ref{ballslem}, we learn that in that case, $X(\tau_i) /
Y(\tau_i)$ cannot be larger than some universal constant, say $C_1$.
In other words, by picking the constant $C$ to be large enough, we can
ensure that
%
%e45 #&#
\begin{equation}
\label{Eixsmall} E_i \mbox{ holds}\quad \Rightarrow\quad X(\tau_i)
< 1.01 \times2^i X_0.
\end{equation}
Otherwise, if $X(\tau_i) \leq1.01 \times2^i X_0$ then we necessarily
have $X(\tau_i) \leq\break 2.02 X(\tau_{i-1})$.
It follows that for all $i \geq1$, either $E_i$ does not hold or
$R(\tau_i) \leq2.02 R(\tau_{i-1})$, and consequently
%
%e46 #&#
\begin{equation}
\label{eqnoei} \EE\bigl[R(\tau_{i}) \mathbf{1}_{E_{i-1}^C} | A(t)
\bigr] \leq C + 2.02 C \leq4C.
\end{equation}
Next, we deal with the case that $E_{i-1}$ holds. By choosing the
constant $C$ to be large enough, we can use Proposition \ref{sumitup}
to get
\[
\EE\bigl[R(T) \mathbf{1}_{E_{i-1}} | A(t) \bigr] \leq\tfrac{1}{2}
\EE\bigl[R(\tau _{i-1}) | A(t) \bigr],
\]
where
\[
T = \min\bigl\{s \geq\tau_{i-1}; X(s) \geq1.3 X(\tau_{i-1})
\bigr\}.
\]
Now, equation (\ref{Eixsmall}) teaches us that
\begin{eqnarray*}
&&\EE\bigl[ R(\tau_i) \mathbf{1}_{E_i} \mathbf{1}_{E_{i-1}}
| A(t) \bigr]\\
&&\qquad \leq
\frac{1.01 \times2}{1.3} \EE\bigl[R(T) | A(t)\bigr] \leq0.9 \EE\bigl[R(\tau
_{i-1}) | A(t) \bigr],
\end{eqnarray*}
and, therefore,
\begin{eqnarray*}
\EE\bigl[ R(\tau_i) \mathbf{1}_{E_{i-1}} | A(t)\bigr]& =& \EE
\bigl[ R(\tau_i) \mathbf {1}_{E_i^C} \mathbf{1}_{E_{i-1}}
| A(t)\bigr] + \EE\bigl[ R(\tau_i) \mathbf {1}_{E_i}
\mathbf{1}_{E_{i-1}} | A(t)\bigr]\\
& \leq&
C + 0.9 \EE\bigl[R(\tau_{i-1}) | A(t) \bigr].
\end{eqnarray*}
Together with (\ref{eqnoei}), we get
\[
\EE\bigl[ R(\tau_i) | A(t) \bigr] \leq5C + 0.9 \EE\bigl[ R(
\tau_{i-1}) | A(t) \bigr].
\]
This completes the proof of the theorem.
\end{pf*}

%s5 #&#
\section{Proof of the main theorem}\label{sec5}

In this section, we finally prove Theorem \ref{mainthm}. We begin with
a lemma which roughly claims that the probability of the aggregate to
intersect a any metric ball whose radius is large enough, is close to
$1$, no matter how far the ball is from the origin of the aggregate.
The proof is a consequence of the tools developed in the previous
section; we show that by choosing a suitable embedding of the aggregate
into the Poincar\'{e} half-plane, the question of intersecting a
specific metric ball boils down to the fact that $Y(t)$ grows rapidly
enough compared to $X(t)$.

%le5.1 #&#
\begin{lemma} \label{lemmain}
There exists a universal constant $R_0 > 0$ such that the following holds:
Given any time $t \geq0$ and any finite starting aggregate, $A(t)$,
which started from a point $p \in\HH$, there exists a number $L>0$
such that for any point $p'$ with $d(p,p') \geq L$ one has
\[
\PP\bigl(A(\infty) \cap B_H\bigl(p', R_0
\bigr) \neq\varnothing | A(t)\bigr) \geq0.99.
\]
\end{lemma}

\begin{pf}
Denote by $D$ the $d_H$-diameter of $A(t)$, and define
\[
M = \max\bigl\{x; \exists y>0 \mbox{ such that } d_H\bigl((x,y),
(0,1)\bigr) < D \bigr\}.
\]
For any two points $p_1,p_2 \in\HH$, there is a (unique up to
orientation) isometric embedding $\phi\dvtx \HH\to\PH$ such that $\phi
(p_1) = (0,1)$ and
$\phi(p_2) = (0,S)$ for some $S \geq1$. So given the starting point
of the aggregate, $p$, and an arbitrary point $p'$ satisfying $d(p,p')
\geq L$ (where $L$ is a constant whose value will be determined later
on), we may therefore assume without loss of generality that $H^{-1}(p)
= (0,1)$ and that $H^{-1}(p') = (0,S)$. Consider the metric ball
\[
B = B_H\bigl((0,S), R_0\bigr).
\]
Clearly, if $R_0$ is a large enough universal constant, this ball will
contain a rectangle of the form
\[
\Psi= [- 3 R_1 S, 3 R_1 S] \times[S, 2 S] \subset\PH,
\]
where $R_1$ is a universal constant whose value will be chosen later on
(see Figure~\ref{figlem} for an illustration). Also consider the
stopping times
\[
\tau_i = \min\bigl\{s; X(s) \geq2^i M \bigr\}\qquad \forall i
\in\mathbb{N}.
\]
By the definition of $M$ we have $X(t) \leq M$, and thus by Theorem
\ref
{mainsec4} we know that for all $i \geq1$,
%
%e47 #&#
\begin{equation}
\label{rdrops1} \EE\bigl[R(\tau_{i+1}) | A(t)\bigr] < C + 0.9 \EE
\bigl[R(\tau_i) | A(t)\bigr]
\end{equation}
%f4 #&#
\begin{figure}

\includegraphics{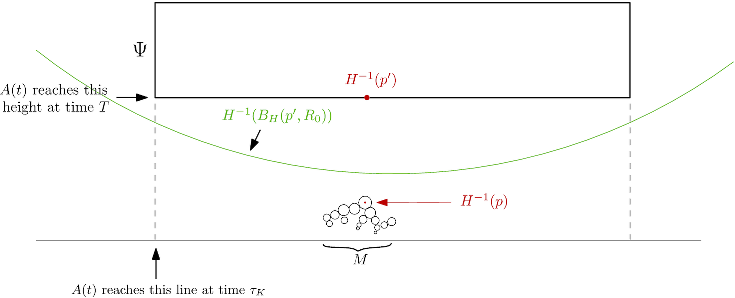}

\caption{The geometric definitions of Lemma \protect\ref{lemmain}.} \label{figlem}
\end{figure}
\hspace*{-3pt}for a universal constant $C>0$, which implies that
%
%e48 #&#
\begin{equation}
\label{rdrops2} \EE\bigl[R(\tau_i) | A(t)\bigr] \geq C_1
\quad\Rightarrow\quad\EE\bigl[R(\tau_{i+1}) | A(t)\bigr] \leq0.95 \EE\bigl[R(
\tau_i) | A(t)\bigr]
\end{equation}
for $C_1 = 20 C$. Next, if $R(\tau_0) > 2 M$, then necessarily by Lemma
\ref{ballslem} it means
that $Y(\tau_0) > c X(\tau_0)$ for a universal constant $c>0$ [since
almost surely only one particle is added at a time, and the increment
in $X(s)$ is not larger than a constant times $Y(s)$]. We deduce that
\begin{eqnarray*}
\EE\bigl[R(\tau_0) | A(t)\bigr] &=& \EE\bigl[R(\tau_0)
\mathbf{1}_{R(\tau_0) < 2
M} | A(t)\bigr] + \EE\bigl[R(\tau_0)
\mathbf{1}_{R(\tau_0) \geq2 M} | A(t)\bigr]\\
& \leq&
2M / Y(t) + c^{-1} \leq2 M + C_2
\end{eqnarray*}
for a universal constant $C_2 > 0$. Together with (\ref{rdrops1}) and
(\ref{rdrops2}), it gives
\[
\EE\bigl[R(\tau_i) | A(t)\bigr] \leq C_1\qquad \forall i
\geq\Theta
\]
with $\Theta= \max (\log_{0.95} \frac{C_1}{ 2 M + C_2}, 1
)$. Denote $K = \lceil\log_2(R_1 S / M) \rceil$.
Recall that we are free to take the constant $L$ as large as we want,
which ensures us that the number $S$ can be as large as we like thanks
to the assumption $d_H(p,p') \geq L$. Now, since the number $M$ does
not depend on the point $p'$ [but only on the aggregate $A(t)$], by
taking $L$ to be large enough, it is legitimate to assume that
\[
K \geq\Theta.
\]
With this assumption, we get
\[
\EE\bigl[R(\tau_K) | A(t)\bigr] \leq C_1,
\]
and also by the definition $\tau_i$,
\[
X(\tau_K) \geq M 2^{K} \geq R_1 S.
\]
These two equations combined yield
\[
\PP\bigl[Y(\tau_K) < S | A(t) \bigr] < C_1 /
R_1 < 0.01,
\]
where the last inequality can be attained by making sure that $R_1$ is
a large enough universal constant (note that the value of $C_1$ has
already been fixed and thus does not depend on $R_1$). Defining
\[
T = \min\bigl\{s; Y(s) \geq S \bigr\},
\]
the previous equation becomes
%
%e49 #&#
\begin{equation}
\label{tauksmall} \PP\bigl(T > \tau_K | A(t) \bigr) < 0.01.
\end{equation}
On the other hand, another application of Lemma \ref{ballslem} with the
fact that only one particle is added at a time almost surely, teaches
us that
\[
X(T \wedge\tau_K) \leq2 R_1 S + C_3 S
\]
for a universal constant $C_3 > 0$, and by choosing that $R_1$ to be
large enough we can assert that
\[
X(T \wedge\tau_K) \leq3 R_1 S
\]
almost surely, without affecting the correctness of the above. Using
the last equation and the definition of $\Psi$, it is easy to check
that we have the implication
\[
T < \tau_K \quad\Rightarrow \quad A(T) \cap\Psi\neq\varnothing\quad\Rightarrow\quad A(T)
\cap B_H\bigl(p', R_0\bigr) \neq
\varnothing.
\]
In light of equation (\ref{tauksmall}), this completes the proof.
\end{pf}

We are finally ready to prove the main theorem.

\begin{pf*}{Proof of Theorem \ref{mainthm}}
The main idea of the proof is to use the previous lemma iteratively, in
order to prove that there exists a random sequence of radii $L_1 \leq
L_2 \leq\cdots $ such that $L_i \to\infty$ almost surely and
a random sequence of stopping times $T_1 \leq T_2 \leq\cdots$ such that
for all $i \geq1$, almost surely
%
%e50 #&#
\begin{equation}
\label{eq23} \qquad\PP \bigl( \# \bigl( A(\infty) \cap B_H(p_0,
L_{i+1}) \bigr) \geq c \operatorname{Vol}_H\bigl(B_H(p_0,
L_{i+1})\bigr) \mid A(T_{i}), L_i \bigr) \geq c,
\end{equation}
where $c>0$ is a universal constant and $p_0$ is the starting point of
the aggregate. This will clearly complete the proof, since it implies
that with probability one there exists a subsequence of radii $\{
L_{i_k}\}_{k=1}^\infty$ such that
\[
\# \bigl(A(\infty) \cap B_H(p_0, L_{i_k})
\bigr) \geq c \operatorname{Vol}_H\bigl(B_H(p_0,
L_{i_k})\bigr) \qquad \forall k \in\mathbb{N}
\]
for a universal constant $c>0$.

We build these sequences inductively. We begin with $L_1 = 1$ and $T_1
= 0$. Suppose $L_i$, $T_i$ and $A(T_i)$ are known. We use the previous
lemma with $A(t) = A(T_i)$ as a starting aggregate. The result of the
lemma ensures the existence of a number $L$ such that
%
%e51 #&#
\begin{equation}
\label{fulldisc} \PP\bigl(A(\infty) \cap B_H(p, R_0)
\neq\varnothing | A(T_i)\bigr) \geq0.99
\end{equation}
for all $p$ such that $d_H(p_0,p) \geq L$. Take $L_{i+1} = \max\{2 L_i,
2L \}$. Now consider a maximal set of disjoint metric balls
of radius $R_0$ whose centers lie within the annulus $B_H(p_0,L_{i+1})
\setminus B_H(p_0,L_{i+1} / 2)$. Denote the centers of these balls by
$p_1,\ldots,p_N$ so that $N$ is the number of balls in this packing. By
the maximality of this set, it is obvious that we have
\[
B_H(p_0,L_{i+1}) \setminus
B_H(p_0,L_{i+1} / 2) \subset\bigcup
_{i=1}^N B_H(p_i, 2
R_0).
\]
Consequently,
%
%e52 #&#
\begin{equation}
\label{Nbig} \qquad N \geq\frac{\operatorname{Vol}_H  ( B_H(p_0,L_{i+1}) \setminus B_H(p_0,L_{i+1} /
2)  )}{\operatorname{Vol}_H(p_0, 2 R_0)} \geq c_1 \operatorname{Vol}_H
\bigl(B_H(p_0, L_{i+1})\bigr)
\end{equation}
for a universal constant $c_1 > 0$. Define
\[
M(t) = \# \bigl\{j \in\{1,\ldots,N \}; B_H(p_j,
R_0) \cap A(t) \neq \varnothing \bigr\}
\]
and note that, since the balls $B_H(p_j, R_0)$ are disjoint, we have that
%
%e53 #&#
\begin{equation}
\label{eqcard} \# \bigl( A(t) \cap B_H(p_0,
L_{i+1}) \bigr) \geq M(t)\qquad \forall t \geq T_i.
\end{equation}
Equation (\ref{fulldisc}) ensures that $\EE[M(\infty) | A(T_i)] \geq
0.99 N$. It then follows from Markov's inequality that
\[
\PP\bigl(M(\infty) > N/2 | A(T_i)\bigr) > \tfrac{1}{2}.
\]
By $\sigma$-additivity, there exists a number $T > 0$ such that
\[
\PP\bigl(M(T) > N/2 | A(T_i)\bigr) > \tfrac{1}{2}.
\]
Set $T_{i+1} = T$. Together with equations (\ref{Nbig}) and (\ref
{eqcard}), this establishes (\ref{eq23}). Note that $L_{i+1}$ and
$T_{i+1}$ only depended on $L_i, T_i$ and $A(T_i)$, and therefore the
conditioning on $A(T_{i})$ and $L_i$ in formula (\ref{eq23}) is
legitimate.

The proof is complete.
\end{pf*}

\section*{Acknowledgements} I would like to thank Itai Benjamini for very
fruitful discussions and for introducing me to the DLA model. I would
also like to thank Yuval Peres and the anonymous referee for their very
useful comments which helped me improve the presentation of this note.

% imsref loaded by akundreckaite, 2014-05-02 13:37:20

%

%\begin{appendix}
%\section{}
%\end{appendix}

%\section*{Acknowledgments}

%\begin{supplement}[id=suppA]
%\slink[doi]{10.1214/00-AOPXXXXSUPP} %[doi,text={...}] - jei reikia
%suskaldyti doi
%\sfilename{aopXXXX\_supp.pdf}
%\end{supplement}

%\begin{thebibliography}{99}
%\bibitem[\protect\citeauthoryear{}{}]{r1}
%\bibitem{r1}
%\end{thebibliography}

\printaddresses

\end{document}